\documentclass[12pt]{article}

\usepackage[T1]{fontenc}
\usepackage{amsthm, amssymb}
\usepackage{mathtools}
\usepackage{mathabx}
\usepackage{dsfont}
\usepackage{accents}
\usepackage{graphicx}
\usepackage{authblk}
\usepackage{xcolor}

\newcommand{\vc}[1]{\boldsymbol{#1}}
\newcommand{\ind}{\mathds{1}}
\newcommand{\inner}[2]{\langle #1, #2\rangle}
\renewcommand{\div}[1]{\text{div}\,(#1)}
\newcommand{\ext}[1]{(#1 dS)^\vee}
\newcommand{\conj}[1]{\overline{#1}\,}
\newcommand{\supp}{\text{supp}\,}
\newcommand{\Rs}{\mathbb{R}}
\newcommand{\XU}[1]{X^{#1}_{\zeta(U,\tau)}}
\newcommand{\Xb}[1]{X^{#1}_{\zeta(1),1/\tau}}

\DeclarePairedDelimiter\abs{\lvert}{\rvert}
\DeclarePairedDelimiter\norm{\lVert}{\rVert}
\DeclarePairedDelimiter\japan{\langle}{\rangle}

\def\Xint#1{\mathchoice
{\XXint\displaystyle\textstyle{#1}}%
{\XXint\textstyle\scriptstyle{#1}}%
{\XXint\scriptstyle\scriptscriptstyle{#1}}%
{\XXint\scriptscriptstyle\scriptscriptstyle{#1}}%
\!\int}
\def\XXint#1#2#3{{\setbox0=\hbox{$#1{#2#3}{\int}$ }
\vcenter{\hbox{$#2#3$ }}\kern-.6\wd0}}

\def\dashint{\Xint-}

\theoremstyle{plain}
  \newtheorem{theorem}{Theorem}
  
  \newtheorem{lemma}[theorem]{Lemma}
  
\theoremstyle{definition}  
  \newtheorem{definition}[theorem]{Definition}
  \newtheorem{example}[theorem]{Example}

\newtheorem{innercustomthm}{Theorem}
\newenvironment{customthm}[1]
  {\renewcommand\theinnercustomthm{#1}\innercustomthm}
  {\endinnercustomthm}

\title{The Bilinear Strategy for Calder\'on's Problem}

\author[ ]{Felipe Ponce-Vanegas.}

\affil[ ]{\small BCAM - Basque Center for Applied Mathematics}
\affil[ ]{\tt fponce@bcamath.org}

\date{}

\begin{document}
\maketitle

\begin{abstract}
Electrical Impedance Imaging would suffer a serious obstruction if for two different conductivities the potential and current measured at the boundary were the same. The Calder\'on's problem is to decide whether the conductivity is indeed uniquely determined by the data at the boundary. In $\Rs^d$, for $d=5,6$, we show that uniqueness holds when the conductivity is in $W^{1+\frac{d-5}{2p}+,p}(\Omega)$, for $d\le p<\infty$. This improves on recent results of Haberman, and of Ham, Kwon and Lee. The main novelty of the proof is an extension of Tao's bilinear Theorem.  
\end{abstract}

\section{Introduction}

Electrical Impedance Imaging is a technique to reconstruct the inner structure of a body from measurements of potential and current at the boundary. At least since the 30', geophysicists have used this technique to identify different layers of earth underground \cite{SBB}. In pioneering work, Calder\'on \cite{MR590275} posed the problem of deciding whether the conductivity is uniquely determined by measurements at the boundary. Calder\'on went on to show uniqueness, roughly, when the conductivity is close to one.

The electrical potential $u$ in a bounded domain $\Omega\subset\Rs^d$ with Lipschitz boundary satisfies the differential equation 
\begin{equation}\label{eq:BVP}
\begin{aligned}
L_{\gamma}u&:=\div{\gamma\nabla u} = 0, \\
u|_{\partial\Omega} &= f,
\end{aligned}
\end{equation}
where $\gamma$ is the conductivity and $f$ the potential at the boundary. We assume that $\gamma\in L^\infty(\Omega)$ and that $\gamma\ge c>0$. If $f\in H^{1/2}(\partial\Omega)$, then a solution $u\in H^1(\Omega)$ exists. The electrical current at the boundary is $\gamma\partial_\nu u\mid_{\partial\Omega}$, where $\nu$ is the outward-pointing normal, and the operator $\Lambda_\gamma: u|_{\partial\Omega}\mapsto\gamma\partial_\nu u\mid_{\partial\Omega}$ is called the Dirichlet-to-Neumann map; we can define the map $\Lambda_\gamma$ rigorously as
\begin{equation}\label{eq:Symmetry_DN}
\inner{\Lambda_{\gamma}f}{g} := \int_{\Omega} \gamma \nabla u\cdot\nabla \bar{v},  
\end{equation}
where $u$ solves \eqref{eq:BVP} and $v\in H^1(\Omega)$ is \textit{any} extension of $g\in H^{1/2}(\partial\Omega)$; hence $\Lambda_{\gamma}:H^{1/2}(\partial\Omega)\mapsto H^{-1/2}(\partial\Omega)$. If we choose $v$ such that $L_{\gamma}v=0$, then we see that $\Lambda_{\gamma}$ is symmetric. Uniqueness fails if two different conductivities $\gamma_1$ and $\gamma_2$ satisfy $\Lambda_{\gamma_1}=\Lambda_{\gamma_2}$; this were the case, for every $f_1,f_2\in H^\frac{1}{2}(\partial\Omega)$ we would have
\begin{equation}\label{eq:Density_Conductivity}
0=\inner{(\Lambda_{\gamma_1}-\Lambda_{\gamma_2})f_1}{f_2}=\int_{\Omega} (\gamma_1-\gamma_2)\nabla u_1\cdot\nabla \conj{u}_2,
\end{equation}
where $L_{\gamma_1}u_1=0$ and $L_{\gamma_2}u_2=0$ are extensions of $f_1$ and $f_2$ respectively. Most of the proofs of uniqueness show that the collection of functions $\{\nabla u_1\cdot\nabla \conj{u}_2\}$ is dense, so $\gamma_1$ and $\gamma_2$ cannot be different. 

Kohn and Vogelius \cite{MR739921} showed that for smooth conductivities $\gamma_1$ and $\gamma_2$, uniqueness holds at the boundary to all orders, so $\partial_\nu^N\gamma_1=\partial_\nu^N\gamma_2$ at $\partial\Omega$ for every integer $N$. In particular, if the conductivities are analytic, then $\gamma_1=\gamma_2$ in $\Omega$. 

In \cite{MR873380}, Sylvester and Uhlmann introduced the method that most of the proofs follow nowadays. If $u_j$ solve the equation \eqref{eq:BVP} for $\gamma_j$, then the function $w_j:=\gamma_j^\frac{1}{2}u_i$ solves the equation $(-\Delta + q_j)w_j=0$ with $q_j=\gamma_j^{-\frac{1}{2}}\Delta\gamma_j^\frac{1}{2}$, and the relationship \eqref{eq:Density_Conductivity} is replaced by
\begin{equation}\label{eq:Density_Laplace}
\int_{\Rs^d}(q_1-q_2)w_1w_2 = 0;
\end{equation}
then, they had to prove that the collection of function $\{w_1w_2\}$ is dense. The integral is evaluated over $\Rs^d$ because the functions $\gamma_1$ and $\gamma_2$ are extended to the whole space, and are arranged so that $\gamma_1=\gamma_2=1$ outside a ball containing $\Omega$. Since $e^{\zeta\cdot x}$ is harmonic when $\zeta\in\mathbb{C}^d$ satisfies $\zeta\cdot\zeta=0$, then they used the ansatz $w_j=e^{\zeta_j\cdot x}(1+\psi_j)$, expecting that $\psi_j$ is somehow negligible for $\abs{\zeta_1},\abs{\zeta_2}\to\infty$. These solutions $w_j$ are called Complex Geometrical Optics (CGO) solutions. Sylvester and Uhlmann selected $\zeta_1$ and $\zeta_2$ such that $\zeta_1+\zeta_2=i\xi$ for $\xi\in\Rs^d$; then, on the assumed smallness of $\psi_j$ for $\abs{\zeta_1},\abs{\zeta_2}\to\infty$, equation \eqref{eq:Density_Laplace} means that $\widehat{q}_1=\widehat{q}_2$, and this implies that $\gamma_1=\gamma_2$. Their argument works well for conductivities in $C^2(\Omega)$. 

In $\Rs^2$, Astala and P\"{a}iv\"{a}rinta \cite{MR2195135} proved that uniqueness holds in $L^\infty(\Omega)$, the best possible result. In higher dimensions, Brown \cite{MR1393424} proved uniqueness for conductivities in $C^{\frac{3}{2}+}(\Omega)$, and this was improved to $W^{\frac{3}{2}, 2d+}(\Omega)$ by Brown and Torres \cite{MR2026763}. By analogy with unique continuation, it is conjectured that the lowest possible regularity is $W^{1,d}(\Omega)$. 

The function $\psi$ in the CGO solution $w=e^{\zeta\cdot x}(1+\psi)$ satisfies the equation
\begin{equation}\label{eq:Definition_Delta_zeta}
\Delta_\zeta \psi:= \Delta\psi+2\zeta\cdot\nabla\psi = q(1+\psi).
\end{equation}
Then, it is necessary to prove that a solution exists and is small. In \cite{MR3024091}, Haberman and Tataru introduced a Bourgain-type space adapted to $p_\zeta(\xi)=-\abs{\xi}^2+2i\zeta\cdot\xi$, the symbol of $\Delta_\zeta$. The space is defined as
\begin{equation*}
\dot{X}^b_{\zeta} :=\{u\mid \norm{u}_{\dot{X}^b_{\zeta}}^2:=\int_{\Rs^d} \abs{p_\zeta(\xi)}^{2b}\abs{\widehat{u}}^2\,d\xi<\infty\},
\end{equation*}
and it follows immediately that $\norm{\Delta_\zeta^{-1}}_{\dot{X}^{-\frac{1}{2}}_{\zeta}\to \dot{X}^\frac{1}{2}_{\zeta}}=1$. The dual of $\dot{X}^b_{\zeta}$ is $\dot{X}^{-b}_{\zeta}$. If we define the multiplication operator $M_q: u\mapsto qu$, then the existence of $\psi$ follows from $\norm{\Delta^{-1}_\zeta M_q}_{\dot{X}^\frac{1}{2}_{\zeta}\to \dot{X}^\frac{1}{2}_{\zeta}}\le \norm{M_q}_{\dot{X}^\frac{1}{2}_{\zeta}\to \dot{X}^{-\frac{1}{2}}_{\zeta}}\le c<1$, and the smallness of $\psi$ follows from the smallness of $\norm{q}_{\dot{X}^{-\frac{1}{2}}_{\zeta}}$. Using these spaces Haberman and Tataru proved uniqueness for Lipschitz conductivities close to one.

Caro and Rogers \cite{MR3456182} proved uniqueness for Lipschitz conductivities without further restriction. They used Carleman estimates, in the spirit of \cite{MR2299741} and \cite{MR2534094}.

After an observation in \cite{NS}, Haberman refined in \cite{MR3397029} the method of Bourgain spaces, and proved uniqueness for conductivities in $W^{1,3}(\Omega)$ for $d=3$, and $W^{1+\frac{d-4}{2p}, p}(\Omega)$ for $p\ge d$ and $d=4,5,6$. He argued as follows: for $\gamma_1$ and $\gamma_2$ he wanted to show that $\norm{M_{q_j}}_{\dot{X}^\frac{1}{2}_{\zeta_j}\to \dot{X}^{-\frac{1}{2}}_{\zeta_j}}$ and $\norm{q_j}_{\dot{X}^{-\frac{1}{2}}_{\zeta_j}}$ are small for some $\zeta_1$ and $\zeta_2$ that satisfy $\zeta_1+\zeta_2=i\xi$, so Haberman proved that there exist sequences $\{\zeta_{1,k}\}$ and $\{\zeta_{2,k}\}$ for which $\norm{M_{q_j}}_{\dot{X}^\frac{1}{2}_{\zeta_{j,k}}\to \dot{X}^{-\frac{1}{2}}_{\zeta_{j,k}}}$ and $\norm{q_j}_{\dot{X}^{-\frac{1}{2}}_{\zeta_{j,k}}}$ tend to zero as $\abs{\zeta_{1,k}},\abs{\zeta_{2,k}}\to\infty$. To find the sequences, he proved that the expected value of both norms goes to zero as $\abs{\zeta_1},\abs{\zeta_2}\to\infty$.

\begin{theorem}[Haberman \cite{MR3397029}]\label{thm:Haberman_vanishing}
Let us write $\zeta(U,\tau):=\tau(Ue_1-iUe_2)$ for $\tau\ge 1$ and $U\in O_d$ a rotation. If $\nabla\log\gamma_1$ and $\nabla\log\gamma_2$ are in $W^{\frac{d-4}{2p}, p}(\Rs^d)$ for $d\le p<\infty$, or in $L^3(\Rs^d)$ for $d=3$, then
\begin{equation*}
\frac{1}{M}\int\limits_M^{2M}\int\limits_{O_d}\norm{M_{q_j}}_{\dot{X}^\frac{1}{2}_{\zeta(U,\tau)}\to \dot{X}^{-\frac{1}{2}}_{\zeta(U,\tau)}}^p\,dUd\tau \text{ and }\frac{1}{M}\int\limits_M^{2M}\int\limits_{O_d}\norm{q_j}_{\dot{X}^{-\frac{1}{2}}_{\zeta(U,\tau)}}^2\,dUd\tau \xrightarrow{M\to\infty} 0.
\end{equation*}
\end{theorem}

The idea is that, when $\abs{\zeta_j}$ is large, the set of bad pairs $(\zeta_1,\zeta_2)$ for which $\norm{M_{q_j}}_{\dot{X}^\frac{1}{2}_{\zeta_{j,k}}\to \dot{X}^{-\frac{1}{2}}_{\zeta_{j,k}}}$ or $\norm{q_j}_{\dot{X}^{-\frac{1}{2}}_{\zeta_{j,k}}}$ is large has measure close to zero, then  it is possible to extract sequences such that these norms are small and such that $\zeta_1+\zeta_2=i\xi$.

The estimates of Haberman are very good, and most of the argument works well just for $\gamma\in W^{1,d}(\Omega)$. The bottle-neck is to get a strong upper bound of $\norm{M_{\partial_i f}}_{\dot{X}^\frac{1}{2}_{\zeta(U,\tau)}\to \dot{X}^{-\frac{1}{2}}_{\zeta(U,\tau)}}$, where $f\in W^{s,p}$ for some $s\ge 0$.  

In Section~\ref{sec:Outline} we proof the next theorem.
\begin{theorem}[Vanishing of the Expected Value]\label{thm:vanishing_Expectation}
Let us write $\zeta(U,\tau):=\tau(Ue_1-iUe_2)$ for $\tau\ge 1$ and $U\in O_d$ a rotation. Suppose that $f$ is a function supported in the unit ball. If $f\in W^{\frac{d-5}{2p}+, p}(\Rs^d)$ for $d\le p < \infty$, then
\begin{equation}
\frac{1}{M}\int\limits_M^{2M}\int\limits_{O_d}\norm{M_{\partial_i f}}_{\dot{X}^\frac{1}{2}_{\zeta(U,\tau)}\to \dot{X}^{-\frac{1}{2}}_{\zeta(U,\tau)}}\,dUd\tau \xrightarrow{M\to\infty} 0.
\end{equation}
\end{theorem}


The main consequence of this theorem is the next improvement on Calder\'on's problem.
\begin{theorem}\label{thm:Main_Thm_Calderon}
For $d=5,6$ suppose that $\Omega\subset\Rs^d$ is a bounded domain with Lipschitz boundary. If $\gamma_1$ and $\gamma_2$ are in $W^{1+\frac{d-5}{2p}+,p}(\Omega)\cap L^\infty$ for $d\le p<\infty$, and if $\gamma_1,\gamma_2\ge c>0$, then
\begin{equation*}
\Lambda_{\gamma_1}=\Lambda_{\gamma_2} \enspace\text{ implies }\enspace \gamma_1=\gamma_2.
\end{equation*}
\end{theorem}

We write $\gamma\in W^{1+\frac{d-5}{2p}+,p}(\Omega)\cap L^\infty$ to emphasize that $\gamma\in L^\infty$, but it follows from Sobolev embedding for domains with Lipschitz boundaries. We note that Theorem~\ref{thm:vanishing_Expectation} holds for $d\ge 3$, and the restriction $d=5,6$ in Theorem~\ref{thm:Main_Thm_Calderon} seems technical; in fact, we can state the following consequence of the vanishing of the expected value.

\begin{theorem}\label{thm:Main_Thm_Calderon_boundary}
For $d\ge 7$ suppose that $\Omega\subset\Rs^d$ is a bounded domain with Lipschitz boundary. If $\gamma_1$ and $\gamma_2$ are in $W^{1+\frac{d-5}{2p}+,p}(\Omega)\cap L^\infty$ for $d\le p<\infty$, if $\partial_\nu\gamma_1=\partial_\nu\gamma_2$ at $\partial\Omega$, and if $\gamma_1,\gamma_2\ge c>0$, then
\begin{equation*}
\Lambda_{\gamma_1}=\Lambda_{\gamma_2} \enspace\text{ implies }\enspace \gamma_1=\gamma_2.
\end{equation*}
\end{theorem}

By the trace theorem the normal derivative $\partial_\nu \gamma$ is well-defined.  The proof of Theorem~\ref{thm:Main_Thm_Calderon} and Theorem~\ref{thm:Main_Thm_Calderon_boundary} has been already summarized in this introduction, and we provide some more details in Section~\ref{sec:Outline}. We refer the reader to the literature to reconstruct the whole argument, in particular to Haberman \cite{MR3397029} and to Ham, Kwon and Lee \cite{HKL}.

\subsection{Restriction Theory}

Ham, Kwon and Lee \cite{HKL} applied deep estimates from restriction theory to improve on Harberman results, and we will follow most of their arguments. We give here a brief introduction to restriction theory and the way it comes in Calder\'on's problem; a detailed exposition of restriction theory can be found in \cite[part IV]{MR3617376}.

We control the norm $\norm{M_{\partial_i f}}_{\dot{X}^\frac{1}{2}_{\zeta(U,\tau)}\to \dot{X}^{-\frac{1}{2}}_{\zeta(U,\tau)}}$ by duality, so we need an upper bound of
\begin{equation}
\inner{(\partial_i f)u}{v}=\int_{\Rs^d}(\partial_i f)u\conj{v}\,dx \quad\text{for } u,v\in \dot{X}^\frac{1}{2}_{\zeta(U,\tau)}.
\end{equation}
The contribution coming from frequencies close the null set of $p_\zeta(\xi)=-\abs{\xi}^2+2i\zeta\cdot\xi$, which we call the characteristic set $\Sigma_\zeta$, is the hardest part we have to deal with. 

The characteristic set $\Sigma_\zeta$ is a $(d-2)$-sphere, and we have to control the duality pairing when the Fourier transform of $u$ and $v$ is concentrated close to $\Sigma_\zeta$. This is just the setting for which restriction theory has been developed; a few classical examples of applications are \cite{MR0320624, MR894584, MR1616917, MR3548534}.

In restriction theory, we seek to prove the best possible bounds $\norm{\widehat{f}\mid_S}_p\le C\norm{f}_q$, where $S$ is a manifold or just a set. One of earliest and most important result is due to Tomas \cite{MR358216} and Stein (unpublished); for the proof see \textit{e.g.} \cite[chp. 9]{MR1232192}.

\begin{theorem}(Tomas-Stein Inequality)
Suppose that $S\subset \Rs^n$ is a compact surface with non-vanishing curvature. If $f\in L^{p}(\Rs^n)$ for $1\le p\le 2\frac{n+1}{n+3}$, then
\begin{equation}
\norm{\widehat{f}\mid_S}_2\le C\norm{f}_p.
\end{equation}
\end{theorem}

The dual operator is called the extension operator, and it is the Fourier transform of a measure $fdS$ supported on the set $S$. The function $\ext{f}$ is the prototype of a function with frequencies highly concentrated close to $S$. In the dual side, the Tomas-Stein inequality is
\begin{equation}
\norm{\ext{f}}_{L^{p'}(\Rs^n)}\le C\norm{f}_{L^2(S)} \quad\text{for }\, 2\frac{n+1}{n-1}\le p'\le\infty.
\end{equation}

Since the earliest days of restriction theory, a kind of stability of bilinear estimates was exploited; for example, the bound $\norm{\ext{f}}_{L^4(\Rs^2)}\le C\norm{f}_2$ is false, but the bound $\norm{(f_1dS_1)^\vee(f_2dS_2)^\vee}_{L^2(\Rs^2)}\le C\norm{f_1}_2\norm{f_2}_2$ is true, whenever the lines $S_1$ and $S_2$ are transversal; curvature is not required. This stability of bilinear estimates was clarified and refined by Tao, Vargas and Vega \cite{MR1625056}.

If we are to expect some improvement of a bilinear estimate, we have to require a separation condition on the surfaces $S_1$ and $S_2$ involved. For example, if $\norm{(f_1dS_1)^\vee(f_2dS_2)^\vee}_{L^2(\Rs^2)}\le C\norm{f_1}_2\norm{f_2}_2$ were true in any case, then just setting $S_1=S_2$ would provide a linear estimate, a false one in this case. One of the key outcomes of \cite{MR1625056} is a general strategy to get linear bounds from bilinear bounds, and we will follow this strategy in Section~\ref{sec:bilinear_Strategy}.

If we are to use the bilinear strategy, we need strong bilinear upper bounds. For some time, the bilinear analogue of the Tomas-Stein inequality in $\Rs^n$, for $n\ge 3$, was known as Klainerman-Machedon conjecture. Wolff made the first big progress, proving the conjecture when the surfaces are subsets of the cone \cite{MR1836285}. Subsequently, Tao refined the method and proved the conjecture when the surfaces are subsets of a surface with positive curvature \cite{MR2033842}. Vargas \cite{MR2106972} and Lee \cite{MR2218987} proved the conjecture when the surfaces are subsets of the hyperboloid, dealing with unusual obstructions.

Since we are interested in the sphere, we need to prove the bilinear theorem for this case. To avoid antipodal points in the bilinear inequality, we restrict ourselves to the surface
\begin{equation}\label{eq:truncated_Sphere}
S:= \{(\xi',\xi_n)\mid \xi_n = 1-\sqrt{1-\abs{\xi'}^2}\text{ and } \abs{\xi'}<\frac{1}{\sqrt{2}}+\frac{1}{10}\}
\end{equation}
Following \cite{MR1625056}, we define also surfaces of elliptic type.

\begin{definition}(Surfaces of Elliptic Type)
A surface $S$ is of $\varepsilon$-elliptic type if:
\begin{itemize}
\item The surface is the graph of a $C^\infty$ function $\Phi:B_1\subset\Rs^{n-1}\to\Rs$.
\item $\Phi(0)=0$ and $\nabla\Phi(0)=0$.
\item The eigenvalues of $D^2\Phi(x)$ lie in $[1-\varepsilon,1+\varepsilon]$ for every $x\in B_1$.
\end{itemize}
\end{definition}

For every $\varepsilon>0$ and for every point in a surface with positive curvature, we can find a sufficiently small neighborhood $U$ so that $U$ is of $\varepsilon$-elliptic type, up to a linear transformation.

We prove in Section~\ref{sec:Bilinear_Estimates} the next extension of Tao's bilinear theorem.

\begin{theorem}[Bilinear Theorem]\label{thm:Bilinear}
Suppose that $S_1,S_2\subset \Rs^n$ are two open subsets of a surface of elliptic type or the hemisphere in \eqref{eq:truncated_Sphere}, and suppose that their diameter is $\lesssim 1$ and they lie at distance $\sim 1$ of each other. If $f_\mu$ and $g_\nu$ are functions with Fourier transforms supported in a $\mu$-neighborhood of $S_1$ and a $\nu$-neighborhood of $S_2$ respectively, for $\mu\le\nu< \mu^\frac{1}{2}<1$, then for every $\delta>0$ it holds that
\begin{equation}\label{eq:Bilinear_Comparable}
\norm{f_\mu g_\nu}_{p'} \le C_\delta\mu^{\frac{n}{2p}-\delta}\nu^{\frac{1}{p}-\delta}\norm{f_\mu}_2\norm{g_\nu}_2, \quad\text{for }1\le p'\le \frac{n}{n-1}. 
\end{equation}
For surfaces of $\varepsilon$-elliptic type, the constant $C_\delta$ may depend on $\varepsilon$ and on the semi-norms $\norm{\partial^N\Phi}_\infty$. The inequalities are best possible in $\mu$ and $\nu$, up to $\delta$-losses.
\end{theorem}

Unexpected phenomena appear: when $\mu$ is much smaller than $\nu$, \textit{i.e.} when $\mu^\frac{1}{2}\le \nu$, then bilinearity does not play any role; moreover, the curvature of the support of $g_\nu$ is of no importance, and the bounds that Tomas-Stein yield cannot be improved. If we try to get bilinear bounds for $f_\mu$ and $g_\nu$ by averaging over translations of the surface and then applying Tao's bilinear theorem, we do not reach the optimal result \eqref{eq:Bilinear_Comparable}, except when $\mu=\nu$.


The reader can consult the symbols and notations we use at the end of the article.

\subsection*{Acknowledgments}

I wish to thank Pedro Caro for his continuous support and stimulating conversations. This research is supported by the Basque Government through the BERC 2018-2021 program, and by the Spanish State Research Agency through BCAM Severo Ochoa excellence accreditation SEV-2017-0718 and through projects ERCEA Advanced Grant 2014 669689 - HADE and PGC2018-094528-B-I00.

\section{Outline of the Proof}\label{sec:Outline}

The proof that Theorem~\ref{thm:vanishing_Expectation} implies Theorem~\ref{thm:Main_Thm_Calderon} is long, and many steps are already well described in the literature. We refer the reader to \cite{MR3397029, HKL} for details.

First, we extend carefully $\gamma_1$ and $\gamma_2$ to the whole space. By the definition of $W^{s,p}(\Omega)$, we can extend $\gamma_1$ to a function in $W^{s,p}(\Rs^d)$. Since $\gamma_j\in W^{1+\frac{d-5}{2p}+,p}(\Omega)$, then by a theorem of Brown in \cite{MR1881563} we have that $\gamma_1=\gamma_2$ at $\partial\Omega$ if $\Lambda_{\gamma_1}=\Lambda_{\gamma_2}$. Now we define the function
\begin{equation*}
\eta :=
\begin{cases}
\gamma_2-\gamma_1 &\text{if } \Omega \\
0 &\text{if } \Omega^c.
\end{cases}
\end{equation*}
Since $\eta$ is zero at $\partial\Omega$ and $\frac{d-5}{2p}+\le \frac{1}{p}$, then $\eta\in W^{1+\frac{d-5}{2p}+,p}(\Rs^d)$ (see \cite[Theorem 1]{MR884984}); this explains the condition $d\le 6$ in Theorem~\ref{thm:Main_Thm_Calderon}. We can thus define the extension $\gamma_2:=\gamma_1+\eta\in W^{1+\frac{d-5}{2p}+,p}(\Rs^d)$. Finally, we arrange the extensions so that $\gamma_1=\gamma_2=1$ outside a ball containing $\Omega$. For $d\ge 7$ we are in the case $\frac{d-5}{2p}+> \frac{1}{p}$, and we need additionally the condition $\partial_\nu\gamma_1=\partial_\nu\gamma_2$ at $\partial\Omega$ to be able to extend the conductivities. This is the condition that we included in Theorem~\ref{thm:Main_Thm_Calderon_boundary}.



 
For all $w_1,w_2\in H^1_{\text{loc}}(\Rs^d)$ that solve  $(-\Delta + q_j)w_j=0$ with $q_j=\gamma_j^{-\frac{1}{2}}\Delta\gamma_j^\frac{1}{2}$, we want to show that the collection of functions $\{w_1w_2\}$ is dense, which implies that $\gamma_1=\gamma_2$; see \cite{MR2026763} for a rigorous justification. Notice that $q_j$ is compactly supported.

For $\zeta_j\cdot \zeta_j=0$, the function $w_j=e^{\zeta_j\cdot x}(1+\psi_j)$ is a CGO solution. The function $\psi_j\in H^1_\text{loc}(\Rs^d)$ has to satisfy the equation 
\begin{equation}\label{eq:psi_equation}
(-\Delta_\zeta + q_j)\psi_j = -q_j.
\end{equation}
If we choose $\zeta_1$ and $\zeta_2$ such that $\zeta_1+\zeta_2=i\xi$ and replace in \eqref{eq:Density_Laplace}, then we get
\begin{multline}\label{eq:FT_q1-q2}
\int_{\Rs^d}(q_1-q_2)e^{i\xi\cdot x} = \int e^{i\xi\cdot x}\psi_2q_2-\int e^{i\xi\cdot x}\psi_1q_1+ \\
+\int e^{i\xi\cdot x}\psi_1 \Delta_\zeta\psi_2 - \int e^{i\xi\cdot x}\psi_2\Delta_\zeta\psi_1 .
\end{multline}
We expect that the functions $\psi_j$ are negligible, so if we ignore them, we would get that $\widehat{q}_1(\xi)=\widehat{q}_2(\xi)$ for every $\xi\in \Rs^d$, which implies $\gamma_1=\gamma_2$.


The space $H^1_\text{loc}(\Rs^d)$ does not seem to be the best suited space to solve \eqref{eq:psi_equation}. Following Haberman and Tataru \cite{MR3024091}, we use the spaces $\dot{X}^b_\zeta$ and $X^b_\zeta$. Since the inclusion $\dot{X}^\frac{1}{2}_{\zeta}\subset H^1_\text{loc}(\Rs^d)$ holds true, then we have
\begin{equation*}
(-\Delta_\zeta+q):\dot{X}^\frac{1}{2}_{\zeta}\to \dot{X}^{-\frac{1}{2}}_{\zeta}.
\end{equation*}
The goal is to find a pair of sequences $\{\zeta_{1,k}\}$ and $\{\zeta_{2,k}\}$ that satisfy the following conditions:
\begin{itemize}
\item $\zeta_{1,k}+\zeta_{2,k}=i\xi$ and $\abs{\zeta_{j,k}}\to\infty$ as $k\to\infty$. 
\item There exist solutions $\psi_{j,k}\in \dot{X}^\frac{1}{2}_{\zeta_{j,k}}$ of the equation \eqref{eq:psi_equation}.
\item $\norm{\psi_{j,k}}_{\dot{X}^\frac{1}{2}_{\zeta_{j,k}}}\to 0$ as $k\to\infty$.
\end{itemize} 

To solve \eqref{eq:psi_equation} we write $(I-\Delta_\zeta^{-1}q)\psi = \Delta_\zeta^{-1}q$. To invert the operator $(I-\Delta_\zeta^{-1}M_q)$, where $M_q:u\mapsto qu$, it suffices to prove that $\norm{M_q}_{\dot{X}^\frac{1}{2}_{\zeta}\to \dot{X}^{-\frac{1}{2}}_{\zeta}}\le c<1$. We also have the upper bound
\begin{equation*}
\norm{\psi}_{\dot{X}^\frac{1}{2}_\zeta} \le \norm{(I-\Delta_\zeta^{-1}M_q)^{-1}}_{\dot{X}^\frac{1}{2}_\zeta\to \dot{X}^\frac{1}{2}_\zeta}\norm{q}_{\dot{X}^{-\frac{1}{2}}_\zeta}\le \frac{1}{1-c}\norm{q}_{\dot{X}^{-\frac{1}{2}}_\zeta}.
\end{equation*}
Then, we can rewrite the goal as: to find a pair of sequences $\{\zeta_{1,k}\}$ and $\{\zeta_{2,k}\}$ that satisfy the following conditions:
\begin{itemize}
\item $\zeta_{1,k}+\zeta_{2,k}=i\xi$ and $\abs{\zeta_{j,k}}\to\infty$ as $k\to\infty$.
\item $\norm{M_{q_j}}_{\dot{X}^\frac{1}{2}_{\zeta_{j,k}}\to \dot{X}^{-\frac{1}{2}}_{\zeta_{j,k}}}\le c<1$ for sufficiently large $k$.
\item $\norm{q_j}_{\dot{X}^{-\frac{1}{2}}_{\zeta_{j,k}}}\to 0$ as $k\to\infty$.
\end{itemize}

To find the sequences $\{\zeta_{1,k}\}$ and $\{\zeta_{2,k}\}$, Haberman proved that the expected value of $\norm{M_{q_j}}_{\dot{X}^\frac{1}{2}_{\zeta}\to \dot{X}^{-\frac{1}{2}}_{\zeta}}$ and $\norm{q_j}_{\dot{X}^{-\frac{1}{2}}_{\zeta}}$ over $\abs{\zeta}\sim M\ge 1$ is small; see Theorem~\ref{thm:Haberman_vanishing}. The reader can see in \cite[sec. 7]{MR3397029} how to find the sequences from the vanishing of the expected value. 

To prove the vanishing of the expected value of $\norm{q_j}_{\dot{X}^{-\frac{1}{2}}_{\zeta}}$, it suffices to assume that $\nabla\log\gamma_j\in L^d(\Rs^d)$, so we will not turn our attention to it.

To control $\norm{M_{q_j}}_{\dot{X}^\frac{1}{2}_{\zeta}\to \dot{X}^{-\frac{1}{2}}_{\zeta}}$ we write $q =\frac{1}{2}\Delta\log\gamma +\frac{1}{4}\abs{\nabla\log\gamma}^2 = \frac{1}{2}\div{\vc{f}}+\frac{1}{2}\abs{\vc{f}}^2$, where the components of $\vc{f}=(f^1,\ldots,f^n)$ are in $W^{s-1,p}(\Rs^d)$. We can divide $M_q$ into terms $M_{\partial_if}$ and $M_{\abs{f}^2}$. Haberman proved that the expected value of $\norm{M_{\abs{f}^2}}_{\dot{X}^\frac{1}{2}_{\zeta}\to \dot{X}^{-\frac{1}{2}}_{\zeta}}$ goes to zero if $f\in L^d(\Rs^d)$, so we are left with  $\norm{M_{\partial_i f}}_{\dot{X}^\frac{1}{2}_{\zeta}\to \dot{X}^{-\frac{1}{2}}_{\zeta}}$.

The estimates for $\norm{M_{\partial_i f}}_{{\dot{X}^\frac{1}{2}_{\zeta}}\to \dot{X}^{-\frac{1}{2}}_{\zeta}}$ are not strong enough to get the vanishing in the limit for $f\in L^d$. To prove Theorem~\ref{thm:vanishing_Expectation}, we assume the following theorem, which we will prove in the next section.

\begin{theorem}\label{thm:Main_Theorem}
Suppose that $f$ is supported in the unit ball. If $f\in W^{\frac{d-5}{2p}+, p}(\Rs^d)$ for $d\le p < \infty$, then
\begin{equation}
\dashint_M\int_{O_d} \norm{M_{\partial_i f}}_{\XU{\frac{1}{2}}\to \XU{-\frac{1}{2}}}\,dUd\tau\le C\norm{f}_{\frac{d-5}{2p}+, p}.
\end{equation}
\end{theorem}

\begin{proof}[Proof of Theorem~\ref{thm:vanishing_Expectation}]
Since $f$ is compactly supported, then $\norm{M_{\partial_i f}}_{\dot{X}^\frac{1}{2}_\zeta\to \dot{X}^{-\frac{1}{2}}_\zeta}\lesssim \norm{M_{\partial_i f}}_{X^\frac{1}{2}_\zeta\to X^{-\frac{1}{2}}_\zeta}$; see \cite[Lemma 2.2(3-4)]{MR3024091}. We estimate $M_g$ by duality as
\begin{equation*}
\abs{\inner{gu}{v}}\le \norm{g}_\infty\norm{u}_2\norm{v}_2\le \frac{1}{\abs{\zeta}}\norm{g}_\infty\norm{u}_{X^\frac{1}{2}_\zeta}\norm{v}_{X^\frac{1}{2}_\zeta}.
\end{equation*}
For some $A\le 1$ to be fixed later, we define $g=P_{\le A}\partial_i f$, where $P_{\le A}$ is the projection to frequencies $\lesssim A$. By Young inequality for convolutions we get
\begin{equation*}
\norm{M_g}_{X^\frac{1}{2}_\zeta\to X^{-\frac{1}{2}}_\zeta}\le \frac{1}{\abs{\zeta}}\norm{g}_\infty \lesssim \frac{A^2}{\abs{\zeta}}\norm{f}_d.
\end{equation*}
The expected value is thus bounded as
\begin{align*}
\dashint\limits_M\int_{O_d} \norm{M_{\partial_i f}}_{\XU{\frac{1}{2}}\to \XU{-\frac{1}{2}}}\,dUd\tau &\lesssim  \frac{A^2}{M}\norm{f}_d+\dashint\limits_M\int_{O_d} \norm{M_{P_{>A}\partial_i f}}_{\XU{\frac{1}{2}}\to \XU{-\frac{1}{2}}}\,dUd\tau \\
&\lesssim \frac{A^2}{M}\norm{f}_d+\norm{P_{>A}f}_{\frac{d-5}{2p}+, p}.
\end{align*}
If we choose $A=M^\frac{1}{4}$ and let $M\to\infty$, then we get the vanishing.
\end{proof}

\section{Estimates for the Expected Value}\label{sec:Average_Theorem}

In this and the next section, we use duality to get an upper bound of $\norm{M_{\partial_j f}}_{X^\frac{1}{2}_{\zeta(U,\tau)}\to X^{-\frac{1}{2}}_{\zeta(U,\tau)}}$ in terms of $f$, $U$ and $\tau$. We want to get an upper bound
\begin{equation}\label{eq:duality_norm_multiplication}
\abs{\inner{(\partial_j f)u}{v}} = \abs{\int_{\Rs^d}(\partial_j f)u\bar{v}\,dx}\le A(U,\tau,f)\norm{u}_{\XU{\frac{1}{2}}}\norm{v}_{\XU{\frac{1}{2}}},
\end{equation}
with a constant $A(U,\tau,f)$ depending on some quantity related to $\norm{f}_{W^{s,p}}$ for $s=\frac{d-5}{2p}+$ and $d\le p<\infty$.

The characteristic set $\Sigma_\zeta$ of $p_\zeta(\xi)=-\abs{\xi}^2+2i\zeta\cdot\xi$, the symbol of $\Delta_\zeta$, is a $(d-2)$-sphere in the hyperplane $\{\xi\mid\inner{Ue_1}{\xi}=0\}$, with center $\tau Ue_2$ and radius $\tau\ge 1$. If $d(\xi,\Sigma)$ denotes the distance from $\xi$ to $\Sigma_\zeta$, then
\begin{equation*}
\abs{p_\zeta(\xi)}\sim
\begin{cases}
\tau d(\xi,\Sigma_\zeta), &\text{for } d(\xi,\Sigma_\zeta)\le\frac{1}{10}\tau,
\\
\tau^2 + \abs{\xi}^2, &\text{for } d(\xi,\Sigma_\zeta)>\frac{1}{10}\tau
\end{cases}
\end{equation*}
We break up the frequencies accordingly into characteristics and non-characteristics, and define the corresponding projections as
\begin{align*}
(Q_lf)^\wedge(\xi) &:=\zeta(\tau^{-1}d(\xi,\Sigma_\zeta))\widehat{f}(\xi) \\
(Q_hf)^\wedge(\xi) &:=(1-\zeta(\tau^{-1}d(\xi,\Sigma_\zeta)))\widehat{f}(\xi),
\end{align*}
where $\zeta\in C_c^\infty(\Rs)$ is supported inside $(-\frac{1}{10}, \frac{1}{10})$. It follows that
\begin{gather}
\norm{Q_h u}_2\le \tau^{-1}\norm{u}_{\XU{\frac{1}{2}}}\label{eq:Bound_X}\\
\norm{\partial_j Q_h u}_2\le \norm{u}_{\XU{\frac{1}{2}}}\label{eq:Bound_derivative_X}.
\end{gather}

In Lemma 3.3 of \cite{MR3397029} Haberman proved, using Tomas-Stein inequality, that
\begin{equation}\label{eq:TS_X}
\norm{u}_\frac{2d}{d-2}\lesssim \norm{u}_{\XU{\frac{1}{2}}}.
\end{equation}
With the help of inequalities \eqref{eq:Bound_X}, \eqref{eq:Bound_derivative_X} and \eqref{eq:TS_X}, we can control in \eqref{eq:duality_norm_multiplication} all the terms involving non-characteristic frequencies. In fact, 
\begin{multline*}
\inner{(\partial_j f)u}{v} = \inner{(\partial_jf)Q_hu}{Q_hv}+\inner{(\partial_jf)Q_hu}{Q_lv}+ \\ 
+\inner{(\partial_jf)Q_lu}{Q_hv}+\inner{(\partial_jf)Q_lu}{Q_lv}.
\end{multline*}
For the first term at the right, after integration by parts, we have
\begin{align}
\abs{\inner{(\partial_j f)Q_hu}{Q_hv}}&\le \norm{f}_d(\norm{\partial_j Q_hu}_2\norm{Q_hv}_\frac{2d}{d-2}+ \notag\\
&\hspace{4cm}+\norm{Q_hu}_\frac{2d}{d-2}\norm{\partial_j Q_hv}_2) \notag\\
&\lesssim \norm{f}_d\norm{u}_{\XU{\frac{1}{2}}}\norm{v}_{\XU{\frac{1}{2}}}.\label{eq:Q_hQ_h}
\end{align}
For the mixed terms we have
\begin{align}
\abs{\inner{(\partial_jf)Q_hu}{Q_lv}}&\le \norm{f}_d(\norm{\partial_jQ_hu}_2\norm{Q_l v}_\frac{2d}{d-2}+ \notag\\
&\hspace{4cm}+\norm{Q_hu}_2\norm{\partial_jQ_lv}_\frac{2d}{d-2}) \notag\\
&\lesssim \norm{f}_d\norm{u}_{\XU{\frac{1}{2}}}\norm{v}_{\XU{\frac{1}{2}}},\label{eq:Q_lQ_h}
\end{align}
where we used the localization of $Q_lv$ to frequencies $\le 5\tau$, so that $\norm{\partial_j Q_lv}_\frac{2d}{d-2}\lesssim \tau\norm{Q_lv}_\frac{2d}{d-2}$; this follows from Young inequality. We are left then with the characteristic frequencies.

We assume that the support of the Fourier transform of $u$ and $v$ lie in a $\frac{1}{10}$-neighborhood of $\Sigma_\zeta$. We define the transformation 
\begin{equation}\label{eq:Dilation_rotation_u}
u_{\tau U}(x):=\tau^{-d}u(\tau^{-1}Ux),
\end{equation}
so that the frequencies of $u_{\tau U}$ are supported in a $\frac{1}{10}$-neighborhood of the $S^{d-2}$ sphere centered at $e_2$ in the hyperplane normal to $e_1$. The Fourier transform of $u_{\tau U}$ is $\widehat{u}_{\tau U}(\xi) = \widehat{u}(\tau U \xi)$, and the $X^{b}_{\zeta(U,\tau)}$-norm scales as
\begin{equation}\label{eq:norm_Change_Variables}
\norm{u}_{X^b_{\zeta(U,\tau)}} = \tau^{\frac{d}{2}+2b}\norm{u_{\tau U}}_{\Xb{b}}.
\end{equation}
We change variables in the pairing \eqref{eq:duality_norm_multiplication} to get
\begin{align}
\inner{(\partial_jf)u}{v} &= \tau^{-d}\int (\partial_jf)(\tau^{-1} Ux)u(\tau^{-1} Ux)\bar{v}(\tau^{-1} Ux)\,dx \notag\\
&= \tau^{2d+1}\int_{B_{\tau}}(\partial_jf_{\tau U})u_{\tau U} \bar{v}_{\tau U}\,dx \notag\\
&= \tau^{2d+1}\inner{(\partial_{Ue_j}f_{\tau U})u_{\tau U}}{v_{\tau U}}, \label{eq:pairing_Change_Variables}
\end{align}
where we used the identity
\begin{equation*}
(\partial_j f)(\tau^{-1} Ux) = \int \xi_j \widehat{f}(\xi) e^{i (\tau^{-1}Ux)\cdot \xi}\,d\xi = \tau^{d+1}(\partial_{Ue_j} f_{\tau U})(x).
\end{equation*}
Therefore, we assume that the characteristic sphere $S^{d-2}$ lies in the normal plane to $e_1$, has radius 1 and is centered at $e_2$. We assume also that the function $f$ is supported in $B_\tau(0)$.

We apply the Hardy-Littlewood decomposition to $f=\sum_{\tau^{-1}\le \lambda}P_\lambda f$, and decompose $u$ and $v$ into dyadic projections $u_\mu$ and $v_\nu$, where $(u_\mu)^\wedge = \zeta(\mu^{-1}d(\xi,\Sigma_\zeta))\widehat{u}$ and $\zeta\in C_c^\infty(\Rs)$ is supported in $(\frac{1}{2},2)$. Then, the pairing \eqref{eq:duality_norm_multiplication} gets into
\begin{align}
\inner{(\partial_w f)u}{v} &=\sum_{\tau^{-1}\le \lambda,\mu,\nu \lesssim 1} \inner{(\partial_wP_{\lambda, \sup(\mu,\nu)}f)u_\mu}{v_\nu} \notag\\
&=\sum_{\substack{\tau^{-1}\le \lambda \lesssim 1 \\ \tau^{-1}\le \mu\le \nu\lesssim 1}} \inner{(\partial_w P_{\lambda, \nu}f)u_\mu}{v_\nu}+\sum_{\substack{\tau^{-1}\le \lambda \lesssim 1 \\ \tau^{-1}\le \mu> \nu\lesssim 1}}\cdots, \label{eq:Decomposition_Distance}
\end{align}
where $\partial_w$ is the derivative in some direction $w$, and $P_{\lambda, \sup(\mu,\nu)}$ is the projection to frequencies $|\xi|\sim \lambda$ and $|\xi_1|\lesssim \sup(\mu,\nu)$. By symmetry, we can assume that $\mu\le\nu$.

We use Toma-Stein to control the low frequency terms, $\lambda\lesssim \nu^\frac{1}{2}$, and the terms with very different characteristic regions, $\mu^\frac{1}{2}\le\nu$.

\begin{theorem}\label{thm:Bilinear_TS}
If $f_\mu$ and $g_\nu$ are functions in $\Rs^n$, and their Fourier transform are supported in a $\mu$- and $\nu$-neighborhood of $S^{n-1}$ respectively, where $\mu\le\nu$, then
\begin{equation}\label{eq:Bilinear_TS}
\norm{f_\mu g_\nu}_{p'}\lesssim \mu^{\frac{n+1}{2p}}\norm{f_\mu}_2\norm{g_\nu}_2, \qquad \text{for } 1\le p'\le \frac{n+1}{n}. 
\end{equation}
\end{theorem}
\begin{proof}
We use H\"{o}lder to get
\begin{equation}\label{eq:proof_TS_bilinear}
\norm{f_\mu g_\nu}_{p'} \le \norm{f_\mu}_{2p'/(2-p')}\norm{g_\nu}_2.
\end{equation}
Since $1\le p'\le \frac{n+1}{n}$ , then $2\le 2p'/(2-p')\le 2\frac{n+1}{n-1}$, and the latter is the Tomas-Stein exponent. To bound the term $\norm{f_\mu}_r$, for $r=\frac{2p'}{2-p'}$, we interpolate between $p'=2$ and $p'=2\frac{n+1}{n-1}$. 

The point $p'=2$ is immediate. For $p'=2\frac{n+1}{n-1}$, we write $\widehat{f}_\mu$ as an average over spheres
\begin{equation*}
f_\mu(x)=\int r^{n-1}\int_{S^{n-1}} \widehat{f}_\mu(r\theta)e(\inner{rx}{\theta})\,d\theta dr := \int r^{n-1}\ext{f_\mu^r}(rx)\,dr
\end{equation*}
We apply Minkowski, Tomas-Stein and Cauchy-Schwarz to find $\norm{f_\mu}_{2\frac{n+1}{n-1}}\le C\mu^\frac{1}{2}\norm{f_\mu}_2$; this leads to
\begin{equation*}
\norm{f_\mu}_r\lesssim \mu^{\frac{n+1}{2}(\frac{1}{2}-\frac{1}{r})}\norm{f_\mu}_2, \qquad \text{for } 2\le r\le 2\frac{n+1}{n-1}.
\end{equation*}
We replace it in \eqref{eq:proof_TS_bilinear} to get
\begin{equation*}
\norm{f_\mu g_\nu}_{p'}\lesssim  \mu^{\frac{n+1}{2p}}\norm{f_\mu}_2\norm{g_\nu}_2,
\end{equation*}
which is what we wanted.
\end{proof}

By H\"{o}lder, we can bound each term in \eqref{eq:Decomposition_Distance} as
\begin{equation}\label{eq:holder_Single_Term}
\abs{\inner{(\partial_wP_{\lambda, \nu}f)u_\mu}{v_\nu}}\le \lambda\norm{P_{\lambda, \nu}f}_p\norm{u_\mu v_\nu}_{p'}.
\end{equation}
To bound the bilinear term, we begin by writing it as
\begin{equation}\label{eq:raw_bilinear_1}
\int |u_\mu v_\nu|^{p'}\,dx = \iint |u_\mu(x_1,\tilde{x}) v_\nu(x_1,\tilde{x})|^{p'}\,d\tilde{x}dx_1.
\end{equation}
We fix $x_1$ as a parameter and define the function $u_\mu^{x_1}(\tilde{x})= u_\mu(x_1,\tilde{x})$; its Fourier transform is the term in parentheses in the formula
\begin{equation*}
u_\mu(x_1,\tilde{x})=\int\Big(\widehat{u}_\mu(x)e^{ix_1\cdot\xi_1}\Big)e^{i\tilde{x}\cdot\tilde{\xi}}\,d\tilde{\xi} = \int \widehat{u}_\mu^{x_1}(\tilde{\xi})e^{i\tilde{x}\cdot\tilde{\xi}}\,d\tilde{\xi}.
\end{equation*}
The support of $\widehat{u}_\mu^{x_1}$ lies in a $\mu$-neighborhood of the sphere $S^{d-2}\subset\Rs^{d-1}$. Hence, we can apply Theorem~\ref{thm:Bilinear_TS} with $n=d-1$ to the inner integral at the right of \eqref{eq:raw_bilinear_1} to get
\begin{equation}\label{eq:Bilinear_with_x_1}
\int |u_\mu v_\nu|^{p'}\,dx \le \mu^{p'\frac{d}{2p}}\int \norm{u_\mu(x_1,\cdot)}_2^{p'}\norm{v_\nu(x_1,\cdot)}_2^{p'}\,dx_1.
\end{equation}
Since $\widehat{u}_\mu$ is supported in the $\mu$-neighborhood of the hyperplane normal to $e_1$, then we can use the formula $u_\mu = u_\mu*_1\phi_\mu$, where $\phi_\mu(x) = \mu\phi(\mu x)$ and $\phi:\Rs\mapsto\Rs_+$ is a smooth function whose Fourier transform equals one in a $\mu$-neighborhood of the origin. Hence, by Minkowski we have
\begin{align*}
\norm{u_\mu(x_1,\cdot)}_2 &= \Big(\int \Big|\int u_\mu(x_1-y_1,\tilde{x})\phi_\mu(y_1)\,dy_1\Big|^2\,d\tilde{x}\Big)^{1/2} \\
&\le \int\norm{u_\mu(x_1-y_1,\cdot)}_2\phi_\mu(y_1)\,dy_1\\
&=(\norm{u_\mu^{z_1}}_{L^2_{\tilde{x}}}*_1\phi_\mu)(x_1).
\end{align*}
This fact and the next lemma allow us to bound the integral at the right of \eqref{eq:Bilinear_with_x_1}.

\begin{lemma}\label{lemma:axis_x1}
Let $a$ and $b$ be two functions in the real line, then
\begin{equation}
\norm{(a*\phi_\mu)b}_{p'}\le C\mu^\frac{1}{p}\norm{a}_2\norm{b}_2, \qquad \text{for } 1\le p'\le 2.
\end{equation}
The inequality is best possible in $\mu$.
\end{lemma}
\begin{proof}
We use H\"{o}lder and Young inequalities to get
\begin{equation*}
\norm{(a*\phi_\mu)b}_{p'}\le \norm{a*\phi_\mu}_{2p'/(2-p')}\norm{b}_2\le \norm{\phi_\mu}_{p'}\norm{a}_2\norm{b}_2,
\end{equation*}
where $\norm{\phi_\mu}_{p'} = \mu^{\frac{1}{p}}\norm{\phi_1}_{p'}$. The example $a=b=\ind_{(-\mu^{-1},\mu^{-1})}$ shows that the constant $\mu^\frac{1}{p}$ is best possible.
\end{proof}

With the aid of Lemma~\ref{lemma:axis_x1} and $\norm{u_\mu}_2\lesssim\mu^{-\frac{1}{2}}\norm{u}_{\Xb{\frac{1}{2}}}$, we continue  \eqref{eq:Bilinear_with_x_1} as
\begin{equation}\label{eq:Total_Bilinear_TS}
\norm{u_\mu v_\nu}_{p'}  \le \mu^\frac{d+2}{2p} \norm{u_\mu}_2\norm{v_\nu}_2\lesssim \mu^{\frac{d+2}{2p}-\frac{1}{2}}\nu^{-\frac{1}{2}}\norm{u}_{\Xb{\frac{1}{2}}}\norm{v}_{\Xb{\frac{1}{2}}}.
\end{equation}
Furthermore, when we are restricted to low frequencies $\lambda\lesssim \nu^\frac{1}{2}$, we can use this bound and \eqref{eq:holder_Single_Term} in the pairing \eqref{eq:Decomposition_Distance} to get for $p=d$
\begin{align}
\abs{\inner{(\partial_wf)u}{v}} &\lesssim \Big(\sum_{\substack{\tau^{-1}\le \lambda \lesssim \nu^\frac{1}{2} \\ \tau^{-1}\le \mu\le \nu\lesssim 1}}\lambda \mu^\frac{1}{d}\nu^{-\frac{1}{2}}\norm{P_\lambda f}_d \Big)\norm{u}_{\Xb{\frac{1}{2}}}\norm{v}_{\Xb{\frac{1}{2}}}+\cdots \notag\\
&\lesssim \Big(\sum_{\tau^{-1}\le \lambda \lesssim \nu^\frac{1}{2}}\lambda \nu^{\frac{1}{d}-\frac{1}{2}}\norm{P_\lambda f}_d \Big)\norm{u}_{\Xb{\frac{1}{2}}}\norm{v}_{\Xb{\frac{1}{2}}} \notag\\
&\lesssim\Big(\sum_{\tau^{-1}\le \lambda \lesssim 1}\lambda^\frac{2}{d}\norm{P_\lambda f}_d \Big)\norm{u}_{\Xb{\frac{1}{2}}}\norm{v}_{\Xb{\frac{1}{2}}} \notag\\
&\lesssim\Big(\sum_{\tau^{-1}\le \lambda \lesssim 1}\norm{P_\lambda f}_d^d \Big)^\frac{1}{d}\norm{u}_{\Xb{\frac{1}{2}}}\norm{v}_{\Xb{\frac{1}{2}}} \notag\\
&\lesssim \norm{f}_d\norm{u}_{\Xb{\frac{1}{2}}}\norm{v}_{\Xb{\frac{1}{2}}}+\cdots.\label{eq:low_frequencies_dual}
\end{align}
On the other hand, when the characteristic frequencies are very different, \textit{i.e.} $\mu^\frac{1}{2}\le\nu$, again by \eqref{eq:Total_Bilinear_TS} and \eqref{eq:holder_Single_Term} in the pairing \eqref{eq:Decomposition_Distance}, we get
\begin{equation}\label{eq:differente_characteristics}
\abs{\inner{(\partial_wf)u}{v}} \le \Big(\sum_{\substack{\nu^\frac{1}{2}\lesssim \lambda \lesssim 1 \\ \tau^{-1}\le \mu\le \nu^2\lesssim 1}}\lambda \mu^{\frac{d+2}{2p}-\frac{1}{2}}\nu^{-\frac{1}{2}}\norm{P_{\lambda,\nu} f}_p \Big)\norm{u}_{\Xb{\frac{1}{2}}}\norm{v}_{\Xb{\frac{1}{2}}} +\cdots.
\end{equation}
We are left thus with the case of high frequencies $\lambda\gtrsim \nu^\frac{1}{2}$, and similar characteristic frequencies $\mu \le\nu\le\mu^\frac{1}{2}$.

\subsection{Bilinear Strategy}\label{sec:bilinear_Strategy}

In this section we assume that $\mu\le\nu<\mu^\frac{1}{2}$, so that the bilinear inequality in Theorem~\ref{thm:Bilinear} give us a small improvement over Tomas-Stein inequality. To pass from bilinear to linear inequalities, we follow the strategy in \cite{MR1625056}.

We decompose the supports of $\widehat{u}_\mu$ and $\widehat{v}_\nu$ into caps of radius $\rho_0\ll 1$ and width $\mu$ and $\nu$ respectively; we number these caps and refer to them by their number $k$. If the vectors normal to two caps $k$ and $k'$ make an angle $\gtrsim \rho_0$, then we call them transversal and denote it by $k\sim k'$; otherwise the caps are not transversal, $k\nsim k'$. For transversal caps we use the bilinear theorem for the sphere. Then, we can write the bilinear term as
\begin{equation*}
u_\mu \conj{v}_\nu = \sum_{k,k'}u_{\mu,k}\bar{v}_{\nu,k'} = \sum_{k\sim k'}u_{\mu,k}\bar{v}_{\nu,k'} + \sum_{k\nsim k'}u_{\mu,k}\bar{v}_{\nu,k'}.
\end{equation*}
Since we cannot apply the bilinear theorem to non-transversal caps, we decompose them again into caps of radius $\rho_1= \frac{1}{2}\rho_0$. If the vectors normal to two caps $k$ and $k'$ make an angle $\sim \rho_1$, then we call them transversal and denote it again by $k\sim k'$; otherwise the caps are not transversal, $k\nsim k'$. For transversal caps we use the bilinear theorem for surfaces of elliptic type, whenever $\rho_0$ is sufficiently small.  We continue this process until $\rho \sim \nu^\frac{1}{2}$, and write
\begin{multline}\label{eq:Total_Bilinear}
\inner{(\partial_wf)u}{v}=\sum_{\substack{\nu^\frac{1}{2}\lesssim \lambda \lesssim 1\\ \mu\le\nu<\mu^\frac{1}{2}}}\Big[ \sum_{\substack{\nu^{1/2}<\rho\lesssim 1 \\ k\sim k'}}\inner{(\partial_wP_{\lambda,\nu}f)u_{\mu,k}^\rho}{v_{\nu,k'}^\rho}+ \\
+\sum_{k\nsim k'}\inner{(\partial_wP_{\lambda,\nu}f)u_{\mu,k}^{\rho^*}}{v_{\nu,k'}^{\rho^*}}\Big],
\end{multline}
where the sum over non-transversal terms is at scale $\rho^*\sim\nu^\frac{1}{2}$. 

\begin{figure}[t]
\centering
\includegraphics[scale=0.9]{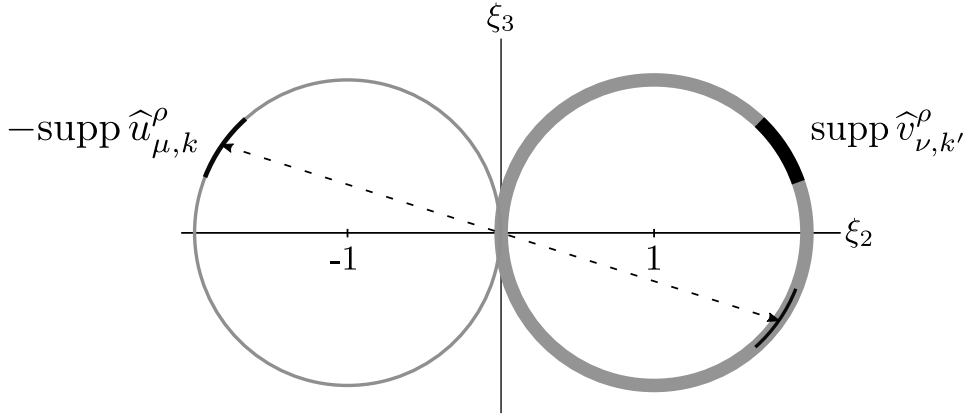}
\caption{The support of two transversal caps is depicted.}\label{fig:supports}
\end{figure}

The support of the inverse Fourier transform of $u_{\mu,k}^\rho \conj{v}_{\nu,k'}^\rho$ has some special properties, and they define when the pairing $\inner{(\partial_wP_{\lambda,\nu}f)u_{\mu,k}^\rho}{v_{\nu,k'}^\rho}$ either vanishes or not. Recall that the support of a convolution $\widecheck{u}_{\mu,k}^\rho*\conj{\widehat{v}}_{\nu,k'}^\rho$ lies in the Minkowski sum of the sets $\supp \widecheck{u}_{\mu,k}^\rho=-\supp \widehat{u}_{\mu,k}^\rho$ and $\supp\widehat{v}_{\nu,k'}^\rho$; see Figure~\ref{fig:supports}. The reader will find easier to evaluate the Minkowski sum of $\supp \widecheck{u}_{\mu,k}^\rho+e_2$ and $\supp\widehat{v}_{\nu,k'}^\rho-e_2$.

We define the support of the caps as 
\begin{equation}
\begin{split}
U_{\mu,k}^\rho &:=\supp \widehat{u}_{\mu,k}^\rho \\
V_{\nu,k}^\rho &:=\supp\widehat{v}_{\nu,k'}^\rho,
\end{split}
\end{equation}
These sets lie in the $\mu$- and $\nu$-neighborhood of a $(d-2)$-sphere. When the caps have radius $\rho_0\ll 1$ and are transversal, then we have that
\begin{equation*}
-U_{\mu,k}^{\rho_0}+V_{\nu,k'}^{\rho_0}\subset \{(\xi_1,\tilde{\xi})\mid\frac{\rho_0}{2}\le \abs{\tilde{\xi}}\le 2-\frac{\rho_0^2}{2},\enspace \abs{\xi_1}\le 2\nu \};
\end{equation*}
Hence, all the terms $\inner{(\partial_w P_{\lambda,\nu}f)u_{\mu,k}^{\rho_0}}{v_{\nu,k'}^{\rho_0}}$ vanish for $\lambda\le c\rho_0$.

When the caps have radius $\rho<\rho_0$, we have to distinguish between neighboring and antipodal caps. Neighboring caps lie in the same ball of radius $2\rho_0$, and antipodal caps lies in different balls of radius $2\rho_0$. We refer to neighboring and antipodal, transversal caps as $k\sim_n k'$ and $k\sim_a k'$ respectively.

If two caps of radius $\nu^\frac{1}{2}\le\rho<\rho_0\ll 1$ are neighboring and transversal, then for the Minkwoski sum we get
\begin{equation*}
-U_{\mu,k}^\rho+V_{\nu,k'}^\rho\subset \{(\xi_1,\tilde{\xi})\mid\abs{\tilde{\xi}}\sim \rho,\enspace \abs{\xi_1}\le 2\nu\}.
\end{equation*}
Hence, only the terms $\inner{(\partial_wP_{\lambda,\nu}f)u_{\mu,k}^\rho}{v_{\nu,k'}^\rho}$ with $\lambda\sim\rho$ survive. When the caps are non-transversal, the Minkowski sum lies in $\{\abs{\tilde{\xi}}\le c\nu^\frac{1}{2}\}$, but we already considered the low frequency terms $\lambda\lesssim \nu^\frac{1}{2}$ in the previous section, so $\inner{(\partial_wP_{\lambda,\nu}f)u_{\mu,k}^{\rho^*}}{v_{\nu,k'}^{\rho^*}}$ always vanishes.

\begin{figure}[t]
\centering
\includegraphics[scale=0.9]{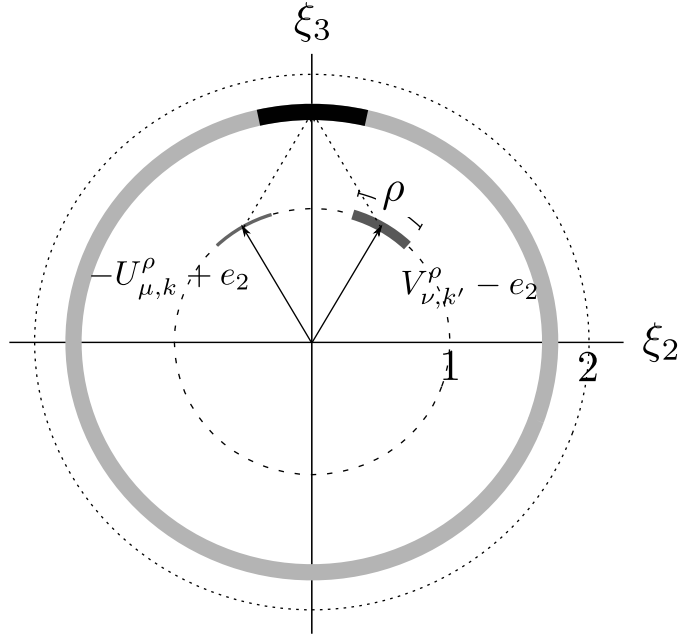}
\caption{The Minkowski sum of two antipodal caps at scale $\rho$.}\label{fig:antipodal}
\end{figure}

If two caps of radius $\nu^\frac{1}{2}\le\rho<\rho_0\ll 1$ are antipodal and transversal, then for the Minkwoski sum we get
\begin{equation*}
-U_{\mu,k}^\rho+V_{\nu,k'}^\rho\subset \{(\xi_1,\tilde{\xi})\mid 2-\abs{\tilde{\xi}}\sim \rho^2,\enspace \abs{\xi_1}\le 2\nu\}.
\end{equation*}
But now we need more detailed information about the support; see Figure~\ref{fig:antipodal}. We can see that $-U_{\mu,k}^\rho+V_{\nu,k'}^\rho$ is a cap of radius $\sim\rho$ in a $\rho^2$-neighborhood of the sphere with radius $2-\rho^2$ and with center at zero; we call it $S_\rho$. Moreover, the caps $-U_{\mu,k}^\rho+V_{\nu,k'}^\rho$ are almost disjoint. In fact, let $x$ be a point in $S_\rho$, $c_k$ be the center of $U_{\mu,k}^\rho$ and $c_{k'}$ be the center of $V_{\nu,k'}^\rho$; if $x$ and $-c_k+e_2$ make an angle $\gtrsim \rho$, then the sum $-c_k+c_k'$ necessarily lies away from $x$. For non-transversal caps, $-U_{\mu,k}^{\rho^*}+V_{\nu,k}^{\rho^*}$ lies in a $\nu$-neighborhood of $S^{d-2}$.  The only terms $\inner{(\partial_wP_{\lambda,\nu}f)u_{\mu,k}^\rho}{v_{\nu,k'}^\rho}$ that survive are for $\lambda\sim 1$.


We will follow the same argument as in the previous section to bound the terms $\inner{(\partial_wP_{\lambda,\nu}f)u_{\mu,k}^\rho}{v_{\nu,k'}^\rho}$; however, the bilinear theorem only holds for well separated caps. To remedy this situation, we use parabolic rescaling.

\begin{theorem}\label{thm:parabolic_rescaling}
Let $f_{\mu,k}$  and $g_{\nu,k'}$ be two functions with Fourier transform supported in a $\mu$- and $\nu$-neighborhood of $S^{n-1}$. If the caps $k$ and $k'$ are transversal at scale $\rho$, then for $1\le p'\le \frac{n+1}{n}$ it holds that
\begin{equation}\label{eq:Parabolic_rescaling}
\begin{gathered}
\norm{f_{\mu,k} g_{\nu,k'}}_{p'}\le C_\varepsilon\rho^{-\frac{1}{p}}\mu^{\frac{n}{2p}-\varepsilon}\nu^{\frac{1}{p}-\varepsilon} \norm{f_{\mu,k}}_2\norm{g_{\nu,k'}}_2 \quad\text{for } \rho> \nu\mu^{-\frac{1}{2}},\\
\norm{f_{\mu,k} g_{\nu,k'}}_{p'} \le C\mu^\frac{n+1}{2}\norm{f_{\mu,k}}_2\norm{g_{\nu,k'}}_2 \quad\text{for } \nu^\frac{1}{2}\le\rho\le \nu\mu^{-\frac{1}{2}}.
\end{gathered}
\end{equation}
\end{theorem}
\begin{proof}
Without loss of generality, we assume that both caps lie in the hypersurface given by the graph of
\begin{equation*}
\varphi(\xi')=1-\sqrt{1-|\xi'|^2}=\frac{1}{2}|\xi'|^2+O(|\xi'|^4),
\end{equation*}
where $\xi = (\xi',\xi_n)\in\Rs^n$; we assume also that the center of the caps are symmetrically placed in the axis $\xi_1$. Since the caps are at distance $\sim \rho$ of each other, after applying the scaling $\xi\mapsto (\rho^{-1}\bar{\xi}, \rho^{-2}\xi_d)$, the support of the new functions $\widehat{F}(\xi):= \widehat{f}_{\mu,k}(\rho\xi', \rho^2\xi_n)$  and $\widehat{G}_{\nu,k'}:=  \widehat{g}_{\nu,k}(\rho\xi', \rho^2\xi_n)$ lie at distance $\sim 1$ of each other, and the hypersurface transforms accordingly to the graph of
\begin{equation*}
\varphi_\rho(\xi'):=\rho^{-2}\varphi(\rho \xi')=\rho^{-2}-\sqrt{\rho^{-4}-|\rho^{-1}\bar\xi|^2}=\frac{1}{2}|\xi'|^2+O(\rho^2_0|\xi'|^4).
\end{equation*}
If $\rho\le\rho_0$ is sufficiently small, then the semi-norms $\norm{\partial^N\varphi_\rho}_\infty$ are uniformly bounded, and the bilinear theorem holds uniformly. The rescaled functions $F$ and $G$ are
\begin{gather*}
F(x) = \rho^{-n-1}f_{\mu,k}(\rho^{-1}x', \rho^{-2}x_n) \\
G(x) = \rho^{-n-1}g_{\nu,k}(\rho^{-1}x', \rho^{-2}x_n).
\end{gather*}
Since the Fourier transforms of $F$ and $G$ are supported now in sets of width $\rho^{-2}\mu$ and $\rho^{-2}\nu$ respectively, then we should apply the bilinear theorem whenever $\rho^{-2}\nu<(\rho^{-2}\mu)^\frac{1}{2}$, and Tomas-Stein otherwise.

If $\rho> \nu\mu^{-\frac{1}{2}}$, then we apply the bilinear theorem to $F$ and $G$ to find
\begin{align*}
\norm{f_{\mu,k} g_{\nu,k'}}_{p'} &= \rho^{2(n+1)-\frac{n+1}{p'}}\norm{FG}_{p'}\\
&\le C_\varepsilon\rho^{2(n+1)-\frac{n+1}{p'}-\frac{n+2}{p}}\mu^{\frac{n}{2p}-\varepsilon}\nu^{\frac{1}{p}-\varepsilon}\norm{F}_2\norm{G}_2 \\
&= C_\varepsilon\rho^{-\frac{1}{p}}\mu^{\frac{n}{2p}-\varepsilon}\nu^{\frac{1}{p}-\varepsilon}\norm{f_{\mu,k}}_2\norm{g_{\nu,k'}}_2;
\end{align*}
if we use Tomas-Stein instead, we get the result for $\rho\le \nu\mu^{-\frac{1}{2}}$
\end{proof}

If we define the quantity
\begin{equation}\label{eq:K_mu_nu}
K^\rho_{\mu,\nu}(p'):=\sup_{\substack{\norm{f_{\mu,k}}_2=1 \\\norm{g_{\nu,k'}}_2=1}}\norm{f_{\mu,k} g_{\nu,k'}}_{p'},
\end{equation}
where the supremum runs over functions $f_{\mu,k}$ and $g_{\nu,k'}$ with Fourier transform supported in caps at scale $\rho$, then we can restate Theorem~\ref{thm:parabolic_rescaling} as
\begin{equation*}
K_{\mu,\nu}^\rho(p')\le 
\begin{cases}
C_\varepsilon\rho^{-\frac{1}{p}}\mu^{\frac{n}{2p}-\varepsilon}\nu^{\frac{1}{p}-\varepsilon} & \text{for } \rho> \nu\mu^{-\frac{1}{2}}\\
C\mu^\frac{n+1}{2p} & \text{for } \nu^\frac{1}{2}\le \rho\le \nu\mu^{-\frac{1}{2}}
\end{cases}
\end{equation*}

By Lemma~\ref{lemma:axis_x1} and Theorem~\ref{thm:parabolic_rescaling}, for $n=d-1$, we get
\begin{align*}
\sum_{k\sim k'}\norm{u_{\mu,k}^\rho v_{\nu,k'}^\rho}_{p'} &\lesssim \mu^\frac{1}{p}K_{\mu,\nu}^\rho \sum_{k\sim k'}\norm{u_{\mu,k}}_2\norm{v_{\nu,k'}}_2 \\
&\lesssim \mu^\frac{1}{p}K_{\mu,\nu}^\rho\norm{u_\mu}_2\norm{v_\nu}_2\\
&\lesssim \mu^{\frac{1}{p}-\frac{1}{2}}\nu^{-\frac{1}{2}}K_{\mu,\nu}^\rho\norm{u}_{\Xb{\frac{1}{2}}}\norm{v}_{\Xb{\frac{1}{2}}}.
\end{align*}
Now, let us consider only neighboring caps at scale $\rho_0$.  By the decomposition in \eqref{eq:Total_Bilinear} we get
\begin{align}
\abs{\inner{(\partial_wf)u}{v}}&\le\sum_{\substack{\nu^\frac{1}{2}\lesssim \lambda \\ \mu\le\nu<\mu^\frac{1}{2}}}\sum_{\substack{\nu^{1/2}\le\rho\sim\lambda \\ k\sim_n k'}}\abs{\inner{(\partial_wP_{\lambda,\nu}f)u_{\mu,k}^\rho}{v_{\nu,k'}^\rho}}+\cdots \notag\\
&\lesssim \sum_{\substack{\nu^\frac{1}{2}\lesssim \lambda \\ \mu\le\nu<\mu^\frac{1}{2}}}\lambda\mu^{\frac{1}{p}-\frac{1}{2}}\nu^{-\frac{1}{2}} \norm{P_{\lambda,\nu}f}_p\sum_{\nu^{1/2}\le\rho\sim\lambda}K_{\mu,\nu}^\rho\norm{u}_{\Xb{\frac{1}{2}}}\norm{v}_{\Xb{\frac{1}{2}}} \notag\\
&\lesssim_\varepsilon \Big(\sum_{\substack{\nu^\frac{1}{2}\lesssim \lambda\le \nu\mu^{-\frac{1}{2}} \\ \mu\le\nu<\mu^\frac{1}{2}}}\lambda\mu^{\frac{d+2}{2p}-\frac{1}{2}}\nu^{-\frac{1}{2}} \norm{P_{\lambda,\nu}f}_p + \notag\\
&\quad\sum_{\substack{\nu\mu^{-\frac{1}{2}}\le \lambda\lesssim 1 \\ \mu\le\nu<\mu^\frac{1}{2}}}\lambda^{1-\frac{1}{p}}\mu^{\frac{d+1}{2p}-\frac{1}{2}}\nu^{\frac{1}{p}-\frac{1}{2}-\varepsilon}\norm{P_{\lambda,\nu}f}_p\Big)\norm{u}_{\Xb{\frac{1}{2}}}\norm{v}_{\Xb{\frac{1}{2}}}. \label{eq:neighboring_caps_estimate}
\end{align}
The operator $P_{\lambda,\nu}$ is the projection to the frequencies $|\xi|\sim \lambda$ and $|\xi_1|\lesssim \nu$. 

When the caps $k$ and $k'$ are antipodal, we have to refine the projection $P_{\lambda,\nu}$, so we project also to the support of $\widecheck{u}_{\mu,k}^\rho*\conj{\widehat{v}}_{\nu,k'}^\rho$ and denote this projection as $P_{\lambda,\nu, k,k'}$. We argue as above to get
\begin{align}
\abs{\inner{(\partial_wf)u}{v}}&\le\sum_{\substack{\lambda\sim 1 \\ \mu\le\nu<\mu^\frac{1}{2}}}\Big(\sum_{\substack{\nu^{1/2}<\rho \\ k\sim_a k'}}\abs{\inner{(\partial_w P_{\lambda,\nu}f)u_{\mu,k}^\rho}{v_{\nu,k'}^\rho}}+\abs{\inner{(\partial_w P_{\lambda,\nu}f)u_{\mu,k}^{\rho^*}}{v_{\nu,k'}^{\rho^*}}}\Big) \notag\\
&\lesssim \sum_{\substack{\lambda\sim 1 \\ \mu\le\nu<\mu^\frac{1}{2}}}\lambda\mu^{\frac{1}{p}-\frac{1}{2}}\nu^{-\frac{1}{2}} \sum_{\nu^{1/2}\le\rho}K_{\mu,\nu}^\rho\sup_{k,k'}\norm{P_{\lambda,\nu,k,k'}f}_p\norm{u}_{\Xb{\frac{1}{2}}}\norm{v}_{\Xb{\frac{1}{2}}}+\cdots. \label{eq:antipodal_caps_estimate}
\end{align}

We have already bounded all the contributions, and we can say that for some functional $A'(f)$ we got an upper bound
\begin{equation*}
\abs{\inner{(\partial_wf)u}{v}}\le (\norm{f}_d+A'(f))\norm{u}_{\Xb{\frac{1}{2}}}\norm{v}_{\Xb{\frac{1}{2}}}
\end{equation*}
If we return to the original variables, and replace $u$ and $v$ by $u_{\tau U}$ and $v_{\tau U}$, and $w$ by $Ue_j$, then by \eqref{eq:Dilation_rotation_u}, \eqref{eq:norm_Change_Variables} and \eqref{eq:pairing_Change_Variables} we get
\begin{align*}
\inner{(\partial_j f)u}{v} &= \tau^{2d+1}\inner{(\partial_w f_{\tau U})u_{\tau U}}{v_{\tau U}} \\
&\lesssim \tau^{2d+1}(\norm{f_{\tau U}}_d+A'(f_{\tau U}))\norm{u_{\tau U}}_{\Xb{1/2}}\norm{v_{\tau U}}_{\Xb{1/2}} \\
&= (\norm{f}_d+\tau^{d-1}A'(f_{\tau U}))\norm{u}_{\dot{X}^{1/2}_{\zeta(\tau, U)}}\norm{v}_{\dot{X}^{1/2}_{\zeta(\tau, U)}}.
\end{align*}
If $m_{\lambda,\nu, k,k'}$ is the multiplier of $P_{\lambda,\nu, k,k'}$, then
\begin{align*}
(P_{\lambda,\nu, k,k'}f_{\tau U})(x) &= (m_{\lambda,\nu, k,k'}(\cdot)\widehat{f}(\tau U\cdot))\,\widecheck{}(x) \\
&= \tau^{-d}(P^{U}_{\tau\lambda,\tau\nu, k,k'}f)(\tau^{-1} U x),
\end{align*}
where the multiplier of $P^U$ is $m(U^{-1}\xi)$. Hence, 
\begin{equation*}
\norm{P_{\lambda,\nu, k,k'}f_{\tau U}}_p = \tau^{-\frac{d}{p'}}\norm{P_{\tau\lambda,\tau\nu, k,k'}^Uf}_p
\end{equation*}
We collect all the estimates \eqref{eq:Q_hQ_h}, \eqref{eq:Q_lQ_h}, \eqref{eq:low_frequencies_dual}, \eqref{eq:differente_characteristics}, \eqref{eq:neighboring_caps_estimate} and \eqref{eq:antipodal_caps_estimate} to conclude this section with the following theorem.
\begin{theorem}\label{thm:Grand_Bound_multiplication}
For $d\le p\le \infty$, the norm of the operator $M_{\partial_jf}:u\in X^\frac{1}{2}_{\zeta(U,\tau)}\mapsto (\partial_jf)u \in X^{-\frac{1}{2}}_{\zeta(U,\tau)}$ has the upper bound
\begin{equation}\label{eq:Multiplication_Norm_Bound}
\norm{M_{\partial_jf}}_{X^{1/2}_{\zeta(\tau, U)}\mapsto X^{-1/2}_{\zeta(\tau, U)}}\lesssim_\varepsilon \norm{f}_d+\tau^{\frac{d}{p}-1}A(\tau, U, f),
\end{equation}
where
\begin{multline}
A(\tau, U, f) := \sum_{\substack{\nu^\frac{1}{2}\lesssim \lambda\lesssim 1 \\ \tau^{-1}\le\mu\le \nu}}Q(\lambda,\mu,\nu)\norm{P_{\tau\lambda,\tau\nu}^Uf}_p + \\
+\sum_{\substack{\lambda\sim 1 \\ \mu\le\nu<\mu^\frac{1}{2}}}\lambda\mu^{\frac{1}{p}-\frac{1}{2}}\nu^{-\frac{1}{2}} \sum_{\nu^{1/2}\le\rho}K_{\mu,\nu}^\rho\sup_{k,k'}\norm{P_{\tau\lambda,\tau\nu,k,k'}^Uf}_p.
\end{multline}
The constant $K_{\mu,\nu}^\rho$ is defined in \eqref{eq:K_mu_nu}, and
\begin{equation*}
Q(\lambda,\mu,\nu) := 
\begin{cases}
\lambda^{1-\frac{1}{p}}\mu^{\frac{d+1}{2p}-\frac{1}{2}}\nu^{\frac{1}{p}-\frac{1}{2}-\varepsilon} &\text{for }\lambda>\nu\mu^{-\frac{1}{2}} \text{ and } \nu\le\mu^\frac{1}{2} \\
\lambda\mu^{\frac{d+2}{2p}-\frac{1}{2}}\nu^{-\frac{1}{2}} &\text{otherwise}. 
\end{cases}
\end{equation*}
\end{theorem} 

\subsection{End of the Proof}

In this section we average the norm $\norm{M_{\partial_jf}}_{\dot{X}^{1/2}_{\zeta(\tau, U)}\mapsto \dot{X}^{-1/2}_{\zeta(\tau, U)}}$ over $\tau$ and $U$. We follow the method of Haberman \cite{MR3397029} and of Ham, Kwon and Lee \cite{HKL}.

By Theorem~\ref{thm:Grand_Bound_multiplication} we have
\begin{multline}
\dashint\limits_M\int_{O_d} \norm{m_{\partial_i f}}_{X^{1/2}_{\zeta(\tau, U)}\mapsto X^{-1/2}_{\zeta(\tau, U)}}\,dUd\tau \lesssim_\varepsilon \norm{f}_d + \\ 
+M^{\frac{d}{p}-1}\sum_{\substack{\nu^\frac{1}{2}\lesssim \lambda\lesssim 1 \\ M^{-1}\le\mu\le \nu}}Q(\lambda,\mu,\nu)\dashint\limits_M\int_{O_d}\norm{P_{\tau\lambda,\tau\nu}^U f}_p\,dUd\tau + \\
+M^{\frac{d}{p}-1}\sum_{\substack{\lambda\sim 1 \\ \mu\le\nu<\mu^\frac{1}{2}}}\lambda\mu^{\frac{1}{p}-\frac{1}{2}}\nu^{-\frac{1}{2}} \sum_{\nu^{1/2}\le\rho}K_{\mu,\nu}^\rho\dashint\limits_M\int_{O_d}\sup_{k,k'}\norm{P_{\tau\lambda,\tau\nu,k,k'}^U f}_p \,dUd\tau. \label{eq:average_multiplication}
\end{multline}
The first average at the right has been already bounded by Haberman.
\begin{lemma}(Haberman, Lemma 5.1 in \cite{MR3397029})\label{lemma:Average_Haberman}
Let $P_{\tau\lambda,\tau\nu}^U$ be the projection to frequencies $\abs{\xi}\sim\tau\lambda$ and to frequencies $\abs{\inner{Ue_1}{\xi}}\le 2\tau\nu$. If $f\in L^p(\Rs^d)$, then
\begin{equation}
\Big(\int_{O_d}\norm{P_{\tau\lambda,\tau\nu}^U f}_p^p\,dU\Big)^\frac{1}{p} \le C\Big(\frac{\nu}{\lambda}\Big)^\frac{1}{p}\norm{f}_p \qquad\text{for } 2\le p \le\infty.
\end{equation}
\end{lemma}
The second average at the right of \eqref{eq:average_multiplication} has been already bounded by Ham, Kwon and Lee.
\begin{lemma}(Ham, Kwon and Lee, Lemma 4.3 in \cite{HKL})\label{lemma:Average_HKL}
Let $k$ and $k'$ denote all transversal, antipodal caps at scale $\rho$, or all non-transversal, antipodal caps at scale $\sim\nu^\frac{1}{2}$. If $P_{\tau\lambda,\tau\nu,k,k'}^U$ is the projection to frequencies $\abs{\xi}\sim\tau\lambda$, $\abs{\inner{Ue_1}{\xi}}\le 2\tau\nu$ and $\xi\in \supp \widecheck{u}_{\mu,k}^\rho*\conj{\widehat{v}}_{\nu,k'}^\rho$, then 
\begin{equation}\label{eq:Multiplication_Antipodal}
\Big(\dashint_M\int_{O_d}\sup_{k,k'}\norm{P_{\tau\lambda,\tau\nu,k,k'}^U f}_p^p\,dUd\tau\Big)^\frac{1}{p} \le C\Big(\frac{\nu}{\lambda}\Big)^\frac{1}{p}\rho^{\frac{2}{p}}\norm{f}_p \qquad\text{for } 2\le p \le\infty.
\end{equation}
\end{lemma}
\begin{proof}[Sketch of the proof]
The proof is by interpolation. For the point $p=\infty$ we get 
\begin{equation*}
\sup_{k,k', U,\tau}\norm{P^U_{\tau\lambda,\tau\nu, k,k'}f}_\infty\lesssim \norm{f}_\infty.
\end{equation*}
For $p=2$ we get
\begin{multline*}
\dashint_M\int_{O_d}\sum_{k,k'}\norm{m^U_{\tau\lambda,\tau\nu, k,k'}\hat{f}}_2^2\,dUd\tau = \\
\int |\hat{f}(\xi)|^2\dashint_M\int_{S^{d-1}}\sum_{k,k'}|m_{\tau\lambda,\tau\nu, k,k'}|^2(|\xi|\omega)\,d\omega d\tau d\xi.
\end{multline*}
The function $\sum_{k,k'}|m_{\tau\lambda,\tau\nu, k,k'}|^2$ is supported in the intersection of an annulus of radius $\tau(2-\rho^2)$ and width $\sim \tau\rho^2$, and a hyperplane of width $\tau\nu$ normal to $Ue_1$. Furthermore, since the sets $-U_{\mu,k}^\rho+V_{\nu,k'}^\rho$ are almost disjoint for different $k$ and $k'$, then $\sum_{k,k'}|m_{\tau\lambda,\tau\nu, k,k'}|^2\lesssim 1$. Hence, for fixed $\xi$ we get
\begin{equation*}
\dashint_M\int_{S^{d-1}}\sum_{k,k'}|m_{\tau\lambda,\tau\nu, k,k'}|^2(|\xi|\omega)\,d\omega d\tau \lesssim \ind_{\{|\xi|\sim M\}}\frac{\nu}{\lambda}\rho^2|\xi|M^{-1},
\end{equation*}
which leads to
\begin{align*}
\dashint_M\int_{O_d}\sum_{k,k'}\norm{m^U_{\tau\lambda,\tau\nu, k,k'}\hat{f}}_2^2\,dUd\tau &\lesssim \frac{\nu}{\lambda}\rho^2\int_{\{|\xi|\sim M\}} |\hat{f}|^2 d\xi \\
& \le \frac{\nu}{\lambda}\rho^2\norm{f}_2^2.
\end{align*}
\end{proof}

We use Lemma~\ref{lemma:Average_Haberman}, Lemma~\ref{lemma:Average_HKL} and H\"older in \eqref{eq:average_multiplication} to get
\begin{align*}
\dashint\int_{O_d} \norm{m_{|D|f}}_{X^{1/2}_{\zeta(\tau, U)}\mapsto X^{-1/2}_{\zeta(\tau, U)}}\,dUd\tau &\lesssim_\varepsilon \norm{f}_d + \\
&\hspace*{-2cm}+M^{\frac{d}{p}-1}\sum_{\substack{\nu^\frac{1}{2}\lesssim \lambda\lesssim 1 \\ M^{-1}\le\mu\le \nu}}Q(\lambda,\mu,\nu)\nu^\frac{1}{p}\lambda^{-\frac{1}{p}}\norm{P_{M\lambda} f}_p + \\
&\hspace{-1.5cm}+M^{\frac{d}{p}-1}\sum_{\substack{\lambda\sim 1 \\ \mu\le\nu<\mu^\frac{1}{2}}}\lambda^{1-\frac{1}{p}}\mu^{\frac{1}{p}-\frac{1}{2}}\nu^{\frac{1}{p}-\frac{1}{2}} \sum_{\nu^{1/2}\le\rho}K_{\mu,\nu}^\rho\rho^{\frac{2}{p}}\norm{P_{M\lambda} f}_p \\
&= \norm{f}_d +\text{I} + \text{II}.
\end{align*}
To bound I, we use the definition of $Q$ in Theorem~\ref{thm:Grand_Bound_multiplication}:
\begin{equation*}
Q(\lambda,\mu,\nu) := 
\begin{cases}
\lambda^{1-\frac{1}{p}}\mu^{\frac{d+1}{2p}-\frac{1}{2}}\nu^{\frac{1}{p}-\frac{1}{2}-\varepsilon} &\text{for }\lambda>\nu\mu^{-\frac{1}{2}} \text{ and } \nu\le\mu^\frac{1}{2} \\
\lambda\mu^{\frac{d+2}{2p}-\frac{1}{2}}\nu^{-\frac{1}{2}} &\text{otherwise}. 
\end{cases}
\end{equation*}
We fix $\lambda\gtrsim M^{-\frac{1}{2}}$, and sum first in $\nu$ and then in $\mu$. Since we assume that $p\ge d\ge 5$, then we get 
\begin{equation*}
\text{I} \le  cM^{\frac{d-5}{2p}+\varepsilon}\sum_{M^{-\frac{1}{2}}\lesssim \lambda\lesssim 1}\lambda^\frac{1}{2}\norm{P_{M\lambda} f}_p\le c\norm{f}_{W^{\frac{d-5}{2p}+\varepsilon,p}}\qquad\text{for } \varepsilon\ll 1.
\end{equation*}
To bound II, recall that:
\begin{equation*}
K_{\mu,\nu}^\rho(p')\le 
\begin{cases}
C_\varepsilon\rho^{-\frac{1}{p}}\mu^{\frac{d-1}{2p}}\nu^{\frac{1}{p}-\varepsilon} & \text{for } \rho> \nu\mu^{-\frac{1}{2}}\\
C\mu^\frac{d}{2p} & \text{for } \nu^\frac{1}{2}\le \rho\le \nu\mu^{-\frac{1}{2}},
\end{cases}
\end{equation*}
We fix $\lambda\sim 1$ and sum in $\rho$, then in $\nu$ and then in $\mu$; the result is
\begin{equation*}
\text{II}\le M^{\frac{d-5}{2p}+\varepsilon}\sum_{\lambda\sim 1}\norm{P_{M\lambda} f}_p\le c\norm{f}_{W^{\frac{d-5}{2p}+\varepsilon,p}}\qquad\text{for } \varepsilon\ll 1.
\end{equation*}
This concludes the proof of Theorem~\ref{thm:Main_Theorem}.

\section{The Bilinear Theorem}\label{sec:Bilinear_Estimates}

In this section we prove the bilinear theorem for two open subsets of the paraboloid. The paraboloid is technically simpler, so the exposition runs more smoothly. After concluding the proof, we explain how we should modify the proof to get Theorem~\ref{thm:Bilinear}. The proof follows closely the ideas presented by Tao in \cite{MR2033842}, and we include here the argument for the sake of completeness. 

\begin{customthm}{\ref{thm:Bilinear}'}
Suppose that $S_1$ and $S_2$ are two open subsets of the paraboloid in $\Rs^n$ with diameter $\lesssim 1$ and at distance $\sim 1$ of each other. If $f_\mu$ and $g_\nu$ are functions with Fourier transforms supported in a $\mu$-neighborhood of $S_1$ and a $\nu$-neighborhood of $S_2$ respectively, for $\mu\le\nu< \mu^\frac{1}{2}<1$, then for every $\varepsilon>0$ it holds that
\begin{equation}\label{eq:Bilinear_Comparable_2}
\norm{f_\mu g_\nu}_{p'} \le C_\varepsilon\mu^{\frac{n}{2p}-\varepsilon}\nu^{\frac{1}{p}-\varepsilon}\norm{f_\mu}_2\norm{g_\nu}_2, \quad\text{for }1\le p'\le \frac{n}{n-1}. 
\end{equation}
The inequalities are best possible, up to $\varepsilon$-losses, in $\mu$ and $\nu$.
\end{customthm}

We can restate the theorem in terms of the quantity
\begin{equation*}
K_{\mu,\nu}(p'):= \sup_{\norm{f_\mu}_2=\norm{g_\nu}_2=1}\norm{f_\mu g_\nu}_{p'}.
\end{equation*}
We get the upper bound of $K_{\mu,\nu}(p')$ by an argument of induction in scales. With some examples, we show that the upper bound $K_{\mu,\nu}(p')$ is the best possible, up to $\varepsilon$-losses.

When $\mu^\frac{1}{2}\le\nu$, the separation between supports does not yield any improvement over Theorem~\ref{thm:Bilinear_TS}, at least in the range $1\le p'\le \frac{n+1}{n}$.

\begin{example}[Case $\mu^\frac{1}{2}\le\nu$]
Let $N_\mu(S_1)$ and $N_\nu(S_2)$ be neighborhoods of two open subsets of the paraboloid with diameter $\sim 1$ and at distance $\sim 1$ of each other. In $N_\mu(S_1)$ let $C_1$ be a cap of radius $\mu^\frac{1}{2}$ and width $\mu$. In $N_\nu(S_2)$ let $C_2:=C_1+a\subset N_\nu(S_2)$ for some vector $a$; this is possible owing to the hypothesis $\mu^\frac{1}{2}\le\nu$. After replacing for $\widehat{u}_\mu=\ind_{C_1}$ and $\widehat{v}_\nu=\ind_{C_2}$ in the bilinear inequality, we get $K_{\mu,\nu}(p')\ge c\mu^\frac{n+1}{2p}$.
\end{example}

Theorem~\ref{thm:Bilinear} holds in $\Rs^2$ without $\varepsilon$-losses. The proof is by averaging over translations of the parabola; see for example Lemma 2.4 in \cite{MR2046812}. 

\begin{example}[Case $\Rs^2$ and $\mu\le\nu\le\mu^\frac{1}{2}$]
Let $N_\mu(S_1)$ and $N_\nu(S_2)$ be separated in the parabola as in Theorem~\ref{thm:Bilinear}'. In $N_\mu(S_1)$ let $C_1$ be a cap of diameter $\nu$ and width $\mu$. In $N_\nu(S_2)$ let $C_2:=C_1+a\subset N_\nu(S_2)$ for some vector $a$. After replacing for $\widehat{u}_\mu=\ind_{C_1}$ and $\widehat{v}_\nu=\ind_{C_2}$ in the bilinear inequality, we get $K_{\mu,\nu}(p')\ge c\mu^\frac{1}{p}\nu^\frac{1}{p}$.
\end{example}

\begin{figure}[t]
\centering
\includegraphics[scale=0.8]{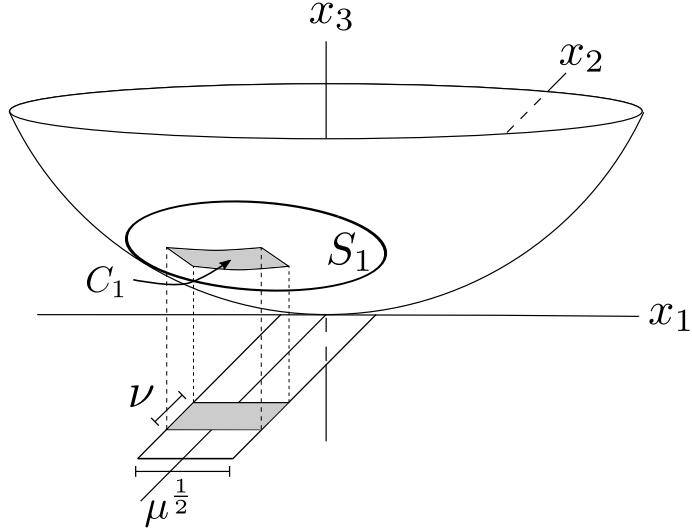}
\caption{The construction of the cap $C_1$.}\label{fig:squashed}
\end{figure}

In higher dimensions we consider as example a modification of the squashed caps in Section 2.7 of \cite{MR1625056}.

\begin{example}[Case $n\ge 3$ and $\mu\le\nu\le\mu^\frac{1}{2}$]
Let $N_\mu(S_1)$ and $N_\nu(S_2)$ be separated in the paraboloid as in Theorem~\ref{thm:Bilinear}'. Let $L_\mu\subset\Rs^{n-1}$ be a $\mu^\frac{1}{2}$-neighborhood of the plane $\{x_1=\cdots=x_{n-2}=0\}$. In $L_\mu$ choose a box $\widetilde{C}_1$ of dimensions $\nu\times\mu^\frac{1}{2}\times\cdots\times\mu^\frac{1}{2}$, so that its lift to the paraboloid lies in $S_1$, and thicken it in $N_\mu(S_1)$ creating so a cap $C_1$ of dimensions $\nu\times\mu^\frac{1}{2}\times\cdots\times\mu^\frac{1}{2}\times\mu$; see Figure~\ref{fig:squashed}. Now, let $C_2:=C_1+a\subset N_\nu(S_2)$ for some vector $a$.  After replacing for $\widehat{u}_\mu=\ind_{C_1}$ and $\widehat{v}_\nu=\ind_{C_2}$ in the bilinear inequality, we get $K_{\mu,\nu}(p')\ge c\mu^\frac{n}{2p}\nu^\frac{1}{p}$.  
\end{example}

The rest of this section is devoted to the proof of the inequality \eqref{eq:Bilinear_Comparable_2} in Theorem~\ref{thm:Bilinear}'. We do first some reductions.

By Galilean and rotational symmetry, we can assume that
\begin{align*}
S_1 &= \{(\xi',\frac{1}{2}\abs{\xi'}^2)\mid \abs{\xi'-c_1e_1}\le c_2 \} \\
S_2 &= \{(\xi',\frac{1}{2}\abs{\xi'}^2)\mid \abs{\xi'+c_1e_1}\le c_2 \};
\end{align*}
the constant $C_\varepsilon$ in \eqref{eq:Bilinear_Comparable_2} depends on $c_1$ and $c_2$. 

It suffices to prove the local inequality
\begin{equation}\label{eq:Local_Bilinear}
\norm{f_\mu g_\nu}_{L^{p'}(B_{\mu^{-1}})} \le C_\varepsilon\mu^{\frac{n}{2p}-\varepsilon}\nu^{\frac{1}{p}-\varepsilon}\norm{f_\mu}_2\norm{g_\nu}_2.
\end{equation} 
In fact, cover $\Rs^n$ with balls $B_{\mu^{-1}}$ and choose a bump function $\zeta_{B_\mu^{-1}}\sim 1$ in $B_{\mu^{-1}}$ so that $\supp \widehat{\zeta}_{B_\mu^{-1}}\subset B_\mu(0)$. Then,
\begin{align*}
\norm{f_\mu g_\nu}_{p'} &\le \sum_{B_{\mu^{-1}}}\norm{f_\mu g_\nu}_{L^{p'}(B_{\mu^{-1}})} \\
&\lesssim \sum_{B_{\mu^{-1}}}\norm{(\widehat{f}_\mu*\widehat{\zeta}_{B_\mu^{-1}})^\vee (\widehat{g}_\nu*\widehat{\zeta}_{B_\mu^{-1}})^\vee}_{L^{p'}(B_{\mu^{-1}})}
\end{align*}
The width of the supports of $\widehat{f}_\mu*\widehat{\zeta}_{B_\mu^{-1}}$ and $\widehat{g}_\nu*\widehat{\zeta}_{B_\mu^{-1}}$ are essentially $\mu$ and $\nu$ respectively. Hence, we can apply the local bilinear inequality \eqref{eq:Local_Bilinear} to get
\begin{align*}
\norm{f_\mu g_\nu}_{p'} &\le C_\varepsilon\mu^{\frac{n}{2p}-\varepsilon}\nu^{\frac{1}{p}-\varepsilon} \sum_{B_{\mu^{-1}}}\norm{f_\mu\zeta_{B_{\mu^{-1}}}}_2\norm{g_\nu\zeta_{B_{\mu^{-1}}}}_2 \\
&\le C_\varepsilon\mu^{\frac{n}{2p}-\varepsilon}\nu^{\frac{1}{p}-\varepsilon}\Big(\sum_{B_{\mu^{-1}}}\norm{f_\mu\zeta_{B_{\mu^{-1}}}}_2^2\Big)^\frac{1}{2}\Big(\sum_{B_{\mu^{-1}}}\norm{g_\nu\zeta_{B_{\mu^{-1}}}}_2^2\Big)^\frac{1}{2} \\
&\le C_\varepsilon\mu^{\frac{n}{2p}-\varepsilon}\nu^{\frac{1}{p}-\varepsilon}\norm{f_\mu}_2\norm{g_\nu}_2,
\end{align*}
which is what we wanted to prove.

At scale $\mu^{-1}$ the function $f_\mu$ looks like $\ext{f}$ for some function $f$ in the paraboloid, so it suffices to prove the next theorem.
\begin{theorem}\label{thm:Bilinear_Modified}
Suppose that $S_1$ and $S_2$ are two open subsets of the paraboloid in $\Rs^n$ with diameter $\sim 1$ and at distance $\sim 1$ of each other. If $fdS$ is a measure supported in $S_1$ and $g_\nu$ a function with Fourier transform supported in a $\nu$-neighborhood of $S_2$, then for $1<R^\frac{1}{2}\le\nu^{-1}\le R$ and for every $\varepsilon>0$ it holds
\begin{equation}\label{eq:Bilinear_Modified}
\norm{\ext{f} g_\nu}_{L^{p'}(B_R)} \le C_\varepsilon R^{\frac{1}{2}(1-\frac{n}{p})+\varepsilon}\nu^{\frac{1}{p}-\varepsilon}\norm{f}_{L^2(S)}\norm{g_\nu}_2,
\end{equation}
where $1\le p'\le \frac{n}{n-1}$.
\end{theorem}
In fact, after a change of variables $\xi\mapsto (\xi',\frac{1}{2}\abs{\xi'}^2+t)$ we can write $f_\mu$ as
\begin{align}
f_\mu(x) &= \int\limits_{-\mu}^{\mu}\Big(\int \widehat{f}_\mu(\xi',\frac{1}{2}\abs{\xi}^2+t)e(\inner{x'}{\xi'}+x_n\frac{1}{2}\abs{\xi'}^2)\,d\xi\Big)\,e(x_nt)\,dt \notag \\
&= \int_{-\mu}^{\mu}\ext{\widehat{f}_{\mu,t}}e(x_nt)\,dt,\label{eq:average_paraboloids}
\end{align}
where $\widehat{f}_{\mu,t}$ is a parabolic slice of $\widehat{f}_\mu$. To bound the local bilinear inequality \eqref{eq:Local_Bilinear} we use Minkowski to get
\begin{equation*}
\norm{f_\mu g_\nu}_{L^{p'}(B_{\mu^{-1}})}\le \int_{-\mu}^{\mu}\norm{\ext{\widehat{f}_{\mu,t}} g_\nu}_{L^{p'}(B_{\mu^{-1}})}\,dt.
\end{equation*} 
Then, writing $\mu^{-1}=R$, we can use Theorem~\ref{thm:Bilinear_Modified} and Cauchy-Schwarz inequality to get
\begin{align*}
\norm{f_\mu g_\nu}_{p'} &\le  C_\varepsilon \mu^{\frac{1}{2}(\frac{n}{p}-1)-\varepsilon}\nu^{\frac{1}{p}-\varepsilon}\int_{-\mu}^{\mu}\norm{f_{\mu,t}}_2\,dt\;\norm{g_\nu}_{2} \\
&\le  C_\varepsilon \mu^{\frac{n}{2p}-\varepsilon}\nu^{\frac{1}{p}-\varepsilon}\norm{f_\mu}_2\norm{g_\nu}_{2}.
\end{align*}
Therefore, we must prove now Theorem~\ref{thm:Bilinear_Modified}. 

The point $p'=1$ of Theorem~\ref{thm:Bilinear_Modified} can be proven readily. By Cauchy-Schwarz and by the trace inequality $\norm{\ext{f}}_2\le CR^\frac{1}{2}\norm{f}_{L^2(S)}$ we get
\begin{equation*}
\norm{\ext{f} g_\nu}_{L^1(B_R)}\le CR^\frac{1}{2}\norm{f}_2\norm{g_\nu}_2.
\end{equation*}
Hence, it suffices to prove the inequality \eqref{eq:Bilinear_Modified} at the point $p'=\frac{n}{n-1}$.

We begin the proof in the next section with the wave packet decomposition. This decomposition is nowadays a classical change of basis, so we only outline it.

\subsection{Wave Packet Decomposition}\label{sec:WP}

Let $f$ be a function in $\Rs^{n-1}$, and decompose the space into caps $\alpha$ of radius $R^{-\frac{1}{2}}$ and center $c_\alpha\in\Rs^{n-1}$. Choose a smooth partition of unity $\{\zeta_\alpha\}$ adapted to the caps $\alpha$ so that $\sum_\alpha\zeta_\alpha^2 = 1$. Use Fourier series adapted to each $\alpha$ to expand $f\zeta_\alpha$ into frequencies $\omega$, and develop $f$ as
\begin{equation*}
f(\xi) = \abs{\alpha}^{-\frac{1}{2}}\sum_{\alpha,\omega}a(\alpha,\omega)\zeta_\alpha(\xi) e(\inner{\omega}{\xi-c_\alpha}),
\end{equation*}
where $\omega = R^\frac{1}{2}\mathbb{Z}^{n-1}$. The coefficients $a$ satisfy the next properties:
\begin{gather}
a(\alpha,\omega) = \frac{1}{\abs{\alpha}^\frac{1}{2}}\int f\zeta_\alpha e(-\inner{\omega}{\xi-c_\alpha})\,d\xi, \label{eq:coefficient_fdS}\\
\sum_{\alpha,\omega}\abs{a(\alpha,\omega)}^2 = \norm{f}_2^2. \label{eq:pitagoras_fdS}
\end{gather}
By the linearity of the extension operator, we can write $\ext{f}$ as
\begin{equation*}
\ext{f}(x)=\sum_{\alpha,\omega}a(\alpha,\omega)\phi_{T(\alpha,\omega)},
\end{equation*}
where $\phi_T$ is a function essentially supported in a tube $T$ of dimensions $R^\frac{1}{2}\times\cdots\times R^\frac{1}{2}\times R$; the angle and position of $T$ are determined by $\alpha$ and $\omega$ respectively. Furthermore,
\begin{equation*}
\abs{\phi_T(x)}\le C_M R^{-\frac{n-1}{2}}\frac{1}{\japan{R^{-\frac{1}{2}}(x'+\omega+x_nc_\alpha)}^M},\quad\text{for } \abs{x_n}\le R;
\end{equation*}
so $\phi_T$ is concentrated in a tube $T$ of direction $(-c_\alpha, 1)$ whose main axis passes through $(-\omega,0)$. We deduce also that for $\delta>0$, for $x\notin R^\delta T$, and for $\abs{x_n}\le R$ it holds
\begin{equation}\label{eq:f_outside_delta}
\abs{\phi_T(x)}\le C_\delta R^{-100n},
\end{equation}
where possibly $C_\delta\to\infty$ as $\delta\to 0$.

The function $g_\nu$ can be written similarly. We decompose $N_\nu(S_2)$ into rectangles $\beta$ of dimensions $\nu\times R^{-\frac{1}{2}}\times\cdots\times R^{-\frac{1}{2}}$ and center $c_\beta\in\Rs^n$, where $c_\beta$ is now a point in $S_2$. Arguing as before we have  
\begin{equation*}
\widehat{g}_\nu(\xi) = \abs{\beta}^{-\frac{1}{2}}\sum_{\beta,\omega}b(\alpha,\omega)\zeta_\beta(\xi) e(\inner{\omega}{\xi-c_\beta}),
\end{equation*}
where $\omega$ belongs to some rotation of the grid $\nu^{-1}\mathbb{Z}\times R^\frac{1}{2}\mathbb{Z}^{n-1}$. Again, we get
\begin{gather}
b(\beta,\omega) = \frac{1}{\abs{\beta}^\frac{1}{2}}\int \widehat{g}_\nu\zeta_\beta e(-\inner{\omega}{\xi-c_\beta})\,d\xi \label{eq:coefficient_g}\\
\sum_{\beta,\omega}\abs{b(\beta,\omega)}^2= \norm{g_\nu}_2^2. \label{eq:pitagoras_g}
\end{gather}
By the linearity of the Fourier transform, we can write $g_\nu$ as
\begin{equation*}
g_\nu = \sum_{\beta,\omega}b(\beta,\omega)\phi_{T(\beta,\omega)},
\end{equation*}
where $T$ are now tubes of dimensions $\nu^{-1}\times R^\frac{1}{2}\times\cdots\times R^\frac{1}{2}$. Again, we get
\begin{gather}
\abs{\phi_T(x)}\le C_M\nu R^{-\frac{n-1}{2}}\frac{1}{\japan{R^{-\frac{1}{2}}\abs{x'+\omega'+x_nc_\beta'}+\nu\abs{x_n+\omega_n}}^M}, \notag \\
\abs{\phi_T(x)}\le C_\delta \nu R^{-100n}, \quad\text{for } x\notin R^\delta T \text{ and for } \delta>0. \label{eq:g_outside_delta}
\end{gather}

We replace the wave packet decomposition into the bilinear inequality \eqref{eq:Bilinear_Modified}, so we must prove that for $\norm{a}_2=1$ and $\norm{b}_2=1$ we have
\begin{equation*}
\norm{\sum_{T_1,T_2}a_{T_1} b_{T_2} \phi_{T_1}\phi_{T_2}}_{L^\frac{n}{n-1}(B_R)}\le C_\varepsilon R^\varepsilon\nu^{\frac{1}{n}-\varepsilon}.
\end{equation*}
Since $\abs{\phi_{T_1}}$ and $\abs{\phi_{T_2}}$ decay strongly outside the tubes, then we can ignore all the tubes that do not intersect the ball $10B_R$, so the number of tubes in each group is $\lesssim R^{Cn}$; recall that $\nu^{-1}\ge R^\frac{1}{2}$. 

Now, for all the terms that satisfy $\abs{a_{T_1}}$ or $\abs{b_{T_2}}\lesssim R^{-C_n}$ the contribution to the bilinear inequality is negligible, so we can ignore all these terms and do pigeonholing in $\abs{a_{T_1}}$ and $\abs{b_{T_2}}$; here, we introduce logarithmic losses. Hence, for two collections of tubes $\mathbb{T}_1$ and $\mathbb{T}_2$ that intersect the ball $10B_R$ we must prove that
\begin{equation}\label{eq:Bilinear_Cleaned}
\norm{\sum_{T_1\in \mathbb{T}_1,T_2\in\mathbb{T}_2} \phi_{T_1}\phi_{T_2}}_{L^\frac{n}{n-1}(B_R)}\le C_\varepsilon R^\varepsilon\nu^{\frac{1}{n}-\varepsilon}\abs{\mathbb{T}_1}^\frac{1}{2}\abs{\mathbb{T}_2}^\frac{1}{2}.
\end{equation}
The proof of this inequality begins with an induction on scales in the next section.

\subsection{Induction on Scales}\label{sec:Induction_Scales}

We want to control the quantity
\begin{equation*}
K_{\nu}(R):=\sup_{\norm{f}_2=\norm{g_\nu}_2=1}\norm{\ext{f}g_\nu}_{L^{p'}(B_R)}.
\end{equation*}
Rough estimates show that $K_{\mu,\nu}(R)$ is finite, thus well defined, and we want to prove that $K_{\mu,\nu}(R)\le C_\varepsilon R^\varepsilon \nu^{\frac{1}{n}-\varepsilon}$.

The induction on scales seeks to control $K_{\mu,\nu}(R)$ in terms of $K_{\mu,\nu}(R^{1-\delta})$ for some $\delta>0$, which we keep fixed in what follows, so we lower scales and stop at scale $\sim\nu^{-1}$, when Tao's bilinear theorem provides the best possible upper bound, up to $\varepsilon$-losses. From now on, we write $R'$ for $R^{1-\delta}$.  

We begin the induction by breaking up the ball $B_R$ into balls $B_{R'}$. Now, we define a relationship between balls and tubes, so that a tube is related to a ball if the contribution of $\phi_T$ to the bilinear term is large in that ball. We need first decompose $B_R$ into balls $q$ of radius $R^\frac{1}{2}$, and now we introduce the following group of definitions for a dyadic number $\mu_2$:
\begin{gather}
\mathbb{T}_2(q):=\{T_2\in\mathbb{T}_2\mid R^\delta T_2\cap q\neq\emptyset\} \label{eq:tubes_q}\\
q(\mu_2):=\{q\subset B_R\mid \mu_2\le \abs{\mathbb{T}_2(q)}<2\mu_2\} \label{eq:multiplicity}\\
\lambda(T_1,\mu_2,B_{R'}):=\abs{\{q\in q(\mu_2)\mid q\subset B_{R'} \text{ and }R^\delta T_1\cap  q\neq\emptyset\}}.
\end{gather}
 
\begin{definition}[Relation between tubes and balls]
For every number $\mu_2$ and every tube $T_1\in\mathbb{T}_1$ choose a ball $B_{R'}^*(\mu_2, T_1)$, if it exists, that satisfies
\begin{equation*}
\lambda(T_1,\mu_2,B^*_{R'}) = \max_{B_{R'}} \lambda(T_1,\mu_2,B_{R'})>0.
\end{equation*} 
We say that a tube $T_1\in\mathbb{T}_1$ is related to a ball $B_{R'}\subset B_R$, or $T_1\sim B_{R'}$, if $B_{R'}\subset 10B_{R'}^*(\mu_2,T_1)$ for some $\mu_2$. The negation of $T_1\sim B_{R'}$ is $T_1\nsim B_{R'}$. Symmetrically, we can define a relation between tubes $T_2\in \mathbb{T}_2$ and balls $B_{R'}$.
\end{definition}

Every tube in $\mathbb{T}_j$ intersects a number $\lesssim R^\delta$ of balls $B_{R'}\subset B_R$, but each tube is related only to $\lesssim \log R$ balls. This follows from the condition $1\le \mu_2\lesssim R^{\frac{n-1}{2}+C\delta}$.

Now, we bound the bilinear term as
\begin{align}
\norm{\sum_{\substack{T_1\in \mathbb{T}_1 \\ T_2\in\mathbb{T}_2}}\phi_{T_1}\phi_{T_2}}_{L^{p'}(B_R)}&\le \sum_{B_{R'}\subset B_R}\norm{\sum_{T_1,T_2} \phi_{T_1}\phi_{T_2}}_{L^{p'}(B_{R'})} \notag\\
&\le \sum_{B_{R'}\subset B_R}\Big(\norm{\sum_{T_1\sim B_{R'},T_2\sim B_{R'}} \phi_{T_1}\phi_{T_2}}_{L^{p'}(B_{R'})} + \notag \\
&+ \norm{\sum_{T_1\nsim B_{R'},T_2} \phi_{T_1}\phi_{T_2}}_{L^{p'}(B_{R'})}+ \norm{\sum_{T_1\sim B_{R'},T_2\nsim B_{R'}} \phi_{T_1}\phi_{T_2}}_{L^{p'}(B_{R'})}\Big). \notag \\
&= \text{I} + \text{II} + \text{III} \label{eq:Bil_simplified}
\end{align}
For the first term I at the right we use the inductive hypothesis, Cauchy-Schwarz, and the bound $\abs{\{B_{R'}\mid T_j\sim B_{R'}\}}\lesssim\log R$ to get
\begin{align}
\sum_{B_{R'}\subset B_R}\norm{\sum_{\substack{T_1\sim B_{R'} \\ T_2\sim B_{R'}}} \phi_{T_1}\phi_{T_2}}_{L^{p'}(B_{R'})} &\le K(R')\sum_{_{B_{R'}}\subset B_R}\abs{\{T_1\sim B_{R'}\}}^\frac{1}{2}\abs{\{T_2\sim B_{R'}\}}^\frac{1}{2} \notag \\
&\le K(R')\Big(\sum_{B_{R'}, T_1}\ind_{\{T_1\sim B_{R'}\}}\Big)^\frac{1}{2}\Big(\sum_{B_{R'}, T_2}\ind_{\{T_2\sim B_{R'}\}}\Big)^\frac{1}{2} \notag \\
&\le C(\log R) K(R')\abs{\mathbb{T}_1}^\frac{1}{2}\abs{\mathbb{T}_2}^\frac{1}{2}. \label{eq:Inductive_Hypothesis}
\end{align}
We have bounded so the main contribution with an acceptable logarithmic loss.

We turn now to II in \eqref{eq:Bil_simplified}; the term III can be similarly controlled, so we will not describe it. We bound the $L^\frac{n}{n-1}$-norm by interpolation between the points $p'=1$ and $p'=2$. For $p'=1$ we use Cauchy-Schwarz and the trace inequality to get
\begin{equation}\label{eq:point_p1}
\norm{\sum_{T_1\nsim B_{R'},T_2} \phi_{T_1}\phi_{T_2}}_{L^1(B_{R'})}\lesssim R^\frac{1}{2}\abs{\mathbb{T}_1}^\frac{1}{2}\abs{\mathbb{T}_2}^\frac{1}{2};
\end{equation}
recall that $\sum_{T_1\nsim B_{R'}}\phi_{T_1} = \ext{f}$ for some function $f$ in $S$, and $\sum_{T_2}\phi_{T_2} = g_\nu$ for some function $g_\nu$, so we only applied the trace theorem to $\ext{f}$, and used \eqref{eq:pitagoras_fdS} and \eqref{eq:pitagoras_g}. We are left with the point $p'=2$.

If we are to prove \eqref{eq:Bilinear_Cleaned} by interpolation, we must get the upper bound
\begin{equation*}
\norm{\sum_{T_1\nsim B_{R'},T_2} \phi_{T_1}\phi_{T_2}}_{L^2(B_{R'})}\lesssim_\delta R^{\frac{1}{2}(1-\frac{n}{2})+C\delta}\nu^\frac{1}{2}\abs{\mathbb{T}_1}^\frac{1}{2}\abs{\mathbb{T}_2}^\frac{1}{2}.
\end{equation*}
This inequality is in general false, if we do not put constrains over the tubes. The simple example $f=1$ and $g_\nu=1$ in $N_\nu(S_2)$ is enough, and worst examples can be given. Hence, we have to exploit the special structure of the tubes $T_1\nsim B_{R'}$.

We use the decomposition of $B_R$ into cubes $q$ of radius $R^\frac{1}{2}$ and the definition \eqref{eq:multiplicity} to write the $L^2$-norm as
\begin{equation*}
\norm{\sum_{T_1\nsim B_{R'},T_2} \phi_{T_1}\phi_{T_2}}_{L^2(B_{R'})}^2 = \sum_{\mu_2}\sum_{q\in q(\mu_2)}\norm{\sum_{T_1\nsim B_{R'},T_2} \phi_{T_1}\phi_{T_2}}_{L^2(q)}^2.
\end{equation*}
By pigeonholing, it suffices to control the norm for a fixed $\mu_2$. We introduce now the definitions 
\begin{gather}
\lambda(T_1,\mu_2):=\abs{\{q\in q(\mu_2)\mid R^\delta T_1\cap  q\neq\emptyset\}} \\
\mathbb{T}_1[\mu_2,\lambda_1]:=\{T_1\in\mathbb{T}_1\mid \lambda_1\le\lambda(T_1,\mu_2)<2\lambda_1\}.
\end{gather}
Since $1\le\lambda_1\lesssim R^{\frac{1}{2}+C\delta}$, by pigeonholing again it suffices to prove
\begin{equation}\label{eq:bilinear_L2_pigeon}
\sum_{q\in q(\mu_2)}\norm{\sum_{\substack{T_1\nsim B_{R'}, T_1\in \mathbb{T}_1[\mu_2,\lambda_1] \\ T_2}} \phi_{T_1}\phi_{T_2}}_{L^2(q)}^2 \lesssim_\delta R^{1-\frac{n}{2}+C\delta}\nu\abs{\mathbb{T}_1}\abs{\mathbb{T}_2}.
\end{equation}
The case $\lambda(T_1,\mu_2)=0$ is handled with \eqref{eq:f_outside_delta}. In the next section, we use the special nature of the $L^2$-norm to decouple the frequencies. 

\subsection{Decoupling at Scale $R^\frac{1}{2}$} \label{sec:decoupling}

We need first a $L^2$ upper bound of the bilinear operator. Recall that the extension operator is defined as
\begin{equation*}
\ext f(x)=\int_{\mathbb{R}^{n-1}} f(\xi)e(\inner{x'}{\xi'}+x_n\varphi(\xi'))\,d\xi',
\end{equation*}
where $\varphi(\xi')=\frac{1}{2}\abs{\xi'}^2$ and $\xi=(\xi',\xi_n)$. For an open subset $S_1$ of the paraboloid, we denote by $\pi(S_1)$ its projection to $\mathbb{R}^{n-1}$.

We need also the Radon transform of a function, and we define it as
\begin{equation*}
R f(\xi',\theta):= \int_{\Rs^{n-1}} f(\xi'+\eta)\delta\big(\inner{\eta}{\theta}\big)\,d\eta;
\end{equation*}
the Radon transform $R f(\xi',\theta)$ is the integral over the hyperplane with normal $\theta$ that passes through $\xi'$.

\begin{lemma}\label{lemma:L2_Bound}
Let $S_1$ and $S_2$ be two open subsets of the paraboloid with radius $\sim 1$ and at distance $\sim 1$ of each other. Suppose that $fdS$ and $gdS$ are measures with support in $S_1$ and $S_2$ respectively. Then, it holds that
\begin{equation}
\norm{\ext{f}\ext{g}}_2^2 \le C\norm{f}_1\sup_{\substack{\xi'\in \pi(S_1) \\ \xi''\in\pi(S_2)}}R\abs{f}\big(\xi',\frac{\xi'-\xi''}{\abs{\xi'-\xi''}}\big)\norm{g}_1\norm{g}_\infty
\end{equation} 
\end{lemma}
\begin{proof}
We compute the square of the extension operator as
\begin{align*}
\abs{\ext{f}(x)}^2 &= \int_{\mathbb{R}^{2(n-1)}} f(\xi_1'+\xi_2')\conj{f}(\xi_2') \\
&\hspace{2cm} e(\inner{x'}{\xi_1'}+x_n(\varphi(\xi_1'+\xi_2')-\varphi(\xi_2')))\,d\xi_1' d\xi_2' \\
&=\int \Big(\int f(\xi_1'+\xi_2')\conj{f}(\xi_2')\delta(\varphi(\xi_1'+\xi_2')-\varphi(\xi_2')-t)\,d\xi_2'\Big)e(\inner{x}{\xi_1})\,d\xi_1 \\
&:= \widecheck{F}(x),
\end{align*}
where $F$ is the function in parentheses. Thus, we get
\begin{equation*}
\norm{\ext{f}\ext{g}}_2^2 =  \int (F*G)^\vee(x)\,dx = (F*G)^\vee(0).
\end{equation*}
We develop the convolution and change variables, so that
\begin{multline*}
\norm{\ext{f}\ext{g}}_2^2 = \int f(\xi_2')\conj{g}(\xi_2'') \\
\int\conj{f}(\xi_2'+\xi_1')g(\xi_2''+\xi_1')\delta(\varphi(\xi_2')-\varphi(\xi_1'+\xi_2')+\xi_{1,n})\delta(\varphi(\xi_1'+\xi_2'')-\varphi(\xi_2'')-\xi_{1,n})\,d\xi_1 \\
\,d\xi_2' d\xi_2''.
\end{multline*}
We can use Fubini to put inside the integral with respect to $\xi_{1,n}$, so that after the change of variables $\xi_{1,n}\mapsto \xi_{1,n}+\varphi(\xi_1'+\xi_2'')-\varphi(\xi_2'')$ we get
\begin{equation}\label{eq:delta_delta}
\begin{split}
I &:= \int \delta(\varphi(\xi_2')-\varphi(\xi_1'+\xi_2')+\xi_{1,n})\delta(\varphi(\xi_1'+\xi_2'')-\varphi(\xi_2'')-\xi_{1,n})\,d\xi_{1,n} \\
&= \delta(\inner{\xi_1'}{\xi_2'-\xi_2''}).
\end{split}
\end{equation}
Then, the $L^2$ norm gets into
\begin{align*}
\norm{\ext{f}\ext{g}}_2^2 &\le \int \abs{f}(\xi_2')\abs{g}(\xi_2'')\int\abs{f}(\xi_2'+\xi_1')\abs{g}(\xi_2''+\xi_1')\delta(\inner{\xi_1'}{\xi_2'-\xi_2''})\,d\xi_1' d\xi_2' d\xi_2'' \\
&\le \norm{f}_1\norm{g}_1\norm{g}_\infty \sup_{\xi_2',\xi_2''}\int\abs{f}(\xi_2'+\xi_1')\delta(\inner{\xi_1'}{\xi_2'-\xi_2''})\,d\xi_1'.
\end{align*}
Finally, by the identity $\delta(at)=a^{-1}\delta(t)$, and the condition of separation between $S_1$ and $S_2$, we get
\begin{equation*}
\int\abs{f}(\xi_2'+\xi_1')\delta(\inner{\xi_1'}{\xi_2'-\xi_2''})\,d\xi_1' \le CR\abs{f}\big(\xi_2',\frac{\xi_2'-\xi_2''}{\abs{\xi_2'-\xi_2''}}\big),
\end{equation*}
which concludes the proof.
\end{proof}

We use now Lemma~\ref{lemma:L2_Bound} to bound each term at the left side of the inequality \eqref{eq:bilinear_L2_pigeon}. To simplify, let us define $\mathbb{T}_1' := \{T_1\nsim B_{R'}\}\cap\mathbb{T}_1[\mu_2,\lambda_1]$. By \eqref{eq:coefficient_fdS} and \eqref{eq:coefficient_g} we can neglect the contribution from tubes such that $R^\delta T\cap q=\emptyset$. We define so the functions
\begin{align*}
f_q(\xi) &:= \abs{\alpha}^{-\frac{1}{2}}\sum_{T_1\in \mathbb{T}_1'(q)}\zeta_\alpha(\xi) e(\inner{\omega}{\xi-c_\alpha}) \\
\widehat{g}_{\nu,q}(\xi) &:= \abs{\beta}^{-\frac{1}{2}}\sum_{T_2\in \mathbb{T}_2(q)}\zeta_\beta(\xi) e(\inner{\omega}{\xi-c_\beta}).
\end{align*}
We write $g_{\nu,q}$ as an average over paraboloids as in \eqref{eq:average_paraboloids}, and by Minkowski and Cauchy-Schwarz we get
\begin{align*}
\norm{\sum_{T_1\in \mathbb{T}_1'(q), T_2\in\mathbb{T}_2(q)} \phi_{T_1}\phi_{T_2}}_{L^2(q)}^2 &\le \norm{\ext{f_q}g_{\nu,q}}_2^2 \\
&\le \norm{\ext{f_q}\int \ext{\widehat{g}_{\nu,q}^t}e(x_n t)\,dt}_2^2 \\
&\le \nu\int\norm{\ext{f_q}\ext{\widehat{g}_{\nu,q}^t}}_2^2\,dt
\end{align*}
We apply Lemma~\ref{lemma:L2_Bound} to the integrand, using the inequalities
\begin{gather*}
\norm{f_q}_1\le R^{-\frac{n-1}{4}}\abs{\mathbb{T}'_1(q)} \\
\norm{\widehat{g}_{\nu,q}^t}_1 \le \nu^{-\frac{1}{2}} R^{-\frac{n-1}{4}}\abs{\mathbb{T}_2(q)}, \qquad \norm{\widehat{g}_{\nu,q}^t}_\infty \le \nu^{-\frac{1}{2}} R^{\frac{n-1}{4}+C\delta},
\end{gather*}
to get
\begin{equation}\label{eq:decouple_1}
\norm{\sum_{\substack{T_1\in \mathbb{T}_1'(q) \\ T_2\in\mathbb{T}_2(q)}} \phi_{T_1}\phi_{T_2}}_{L^2(q)}^2 \le C\nu R^{-\frac{n-1}{4}+C\delta}\abs{\mathbb{T}'_1(q)}\abs{\mathbb{T}_2(q)}\sup_{\substack{\xi'\in \pi(S_1) \\ \xi''\in\pi(S_2)}}R\abs{f_q}\big(\xi',\frac{\xi'-\xi''}{\abs{\xi'-\xi''}}\big).
\end{equation}

Let $\mathbb{T}_1'(q)(\xi',\xi'-\xi'')$ denote the collection of tubes in $\mathbb{T}_1'(q)$ such that the corresponding cap $\alpha$ intersects the hyperplane with normal $(\xi'-\xi'')/\abs{\xi'-\xi''}$ that passes through $\xi'$. Then,
\begin{align*}
\sup_{\substack{\xi'\in \pi(S_1) \\ \xi''\in\pi(S_2)}}R\abs{f_q}\big(\xi',\frac{\xi'-\xi''}{\abs{\xi'-\xi''}}\big) &\le R^{-\frac{n-1}{4}+\frac{1}{2}}\sup_{\substack{\xi'\in \pi(S_1) \\ \xi''\in\pi(S_2)}}\abs{\mathbb{T}_1'(q)(\xi',\xi'-\xi'')} \\
&:= R^{-\frac{n-1}{4}+\frac{1}{2}}\nu(q,\mu_2,\lambda_1);
\end{align*}
we choose the last definition with the same notation as Tao in \cite{MR2033842}. We replace in \eqref{eq:decouple_1} to find
\begin{equation*}
\norm{\sum_{T_1\in \mathbb{T}_1'(q), T_2\in\mathbb{T}_2(q)} \phi_{T_1}\phi_{T_2}}_{L^2(q)}^2 \le C\nu R^{1-\frac{n}{2}+C\delta}\nu(q,\mu_2,\lambda_1)\abs{\mathbb{T}'_1(q)}\abs{\mathbb{T}_2(q)},
\end{equation*}
where $\mathbb{T}_1' := \{T_1\nsim B_{R'}\}\cap\mathbb{T}_1[\mu_2,\lambda_1]$. Summing over all the cubes $q\in q(\mu_2)$ we get
\begin{equation}\label{eq:bilinear_decoupled}
\sum_{q\in q(\mu_2)}\norm{\sum_{T_1\in \mathbb{T}_1'(q), T_2\in\mathbb{T}_2(q)} \phi_{T_1}\phi_{T_2}}_{L^2(q)}^2 \le C\nu R^{1-\frac{n}{2}+C\delta}\sum_{q\in q(\mu_2)}\nu(q,\mu_2,\lambda_1)\abs{\mathbb{T}'_1(q)}\abs{\mathbb{T}_2(q)}.
\end{equation}
The term at the right does not involve oscillations, so we achieved a decoupling of the oscillating tubes at the left. To conclude the proof of \eqref{eq:bilinear_L2_pigeon}, we must get an upper bound of $\nu(q,\mu_2,\lambda_1)$, which we do in the next section.
 
\subsection{A Kakeya-type Estimate} \label{sec:Kakeya}

In this section we aim to prove the inequality
\begin{equation}\label{eq:nu_term}
\nu(q_0,\mu_2,\lambda_1)\lesssim R^{C\delta}\frac{\abs{\mathbb{T}_2}}{\mu_2\lambda_1},
\end{equation}
for some fixed $q_0\in q(\mu_2)$, $\mu_2$ and $\lambda_1$. For any $\xi'\in\pi(S_1)$ and $\xi''\in\pi(S_2)$ we consider then the following bilinear expression
\begin{equation*}
B:=\int\limits_{\substack{q\in q(\mu_2) \\ B_R\backslash 10B_{R'}}} \sum_{T_1\in\mathbb{T}'_1(q_0)(\xi',\xi'-\xi'')}\ind_{2R^\delta T_1}\sum_{T_2\in \mathbb{T}_2}\ind_{2R^\delta T_2}.
\end{equation*}

By the definition of $q(\mu_2)$ we get
\begin{equation*}
B\gtrsim \mu_2\sum_{T_1\in\mathbb{T}'_1(q_0)(\xi',\xi'-\xi'')}\int_{\substack{q\in q(\mu_2) \\ B_R\backslash 10B_{R'}}} \ind_{2R^\delta T_1}.
\end{equation*}
Since for $T_1\in \{T_1\nsim B_{R'}\}\cap\mathbb{T}_1[\mu_2,\lambda_1]$ it holds that $\abs{\{q\in q(\mu_2)\mid R^\delta T_1\cap q \neq \emptyset\}}\sim\lambda_1$, we see that 
\begin{equation*}
\abs{\{q\in q(\mu_2)\mid q\subset B_R\backslash 10B_{R'}\text{ and } R^\delta T_1\cap q \neq \emptyset\}}\gtrsim R^{-\delta}\lambda_1.
\end{equation*}
Then,
\begin{equation}\label{eq:B_lower}
B \gtrsim R^{\frac{n}{2}-C\delta}\lambda_1\mu_2 \abs{\mathbb{T}'_1(q_0)(\xi',\xi'-\xi'')}
\end{equation}

To get an upper bound of $B$, we re-order the summations so that
\begin{equation*}
B \le\sum_{T_2\in\mathbb{T}_2}\int_{B_R\backslash 10B_{R'}} \ind_{2R^\delta T_2}\sum_{T_1\in\mathbb{T}'_1(q_0)(\xi',\xi'-\xi'')}\ind_{2R^\delta T_1}.
\end{equation*}
Since all the tubes intersect $q_0\subset B_{R'}$, we see that
\begin{equation*}
\sum_{T_1\in\mathbb{T}'_1(q_0)(\xi',\xi'-\xi'')}\ind_{2R^\delta T_1}(x)\lesssim R^{C\delta} \quad\text{for } x\in B_R\backslash 10B_{R'}.
\end{equation*}
The tubes in $\mathbb{T}'_1(q_0)(\xi',\xi'-\xi'')$ have directions $(-c_\alpha,1)$, where $c_\alpha$ lies at distance $<R^{-\frac{1}{2}}$ from a hyperplane with normal direction $\xi'-\xi''$ that passes through $\xi'$. Then, the main axis of all the tubes in $\mathbb{T}'_1(q_0)(\xi',\xi'-\xi'')$ make an angle $<R^{-\frac{1}{2}}$ with a hyperplane with normal direction $(\xi'-\xi'', \inner{\xi'}{\xi'-\xi''})$ that passes through $q_0$. It amounts to saying that the support of $\sum_{T_1\in\mathbb{T}'_1(q_0)(\xi',\xi'-\xi'')}\ind_{2R^\delta T_1}$ lies inside the $R^{\frac{1}{2}+\delta}$-neighborhood of a hyperplane that passes through $q_0$. Furthermore, every tube from $\mathbb{T}_2$ intersects the hyperplane transversally, making an angle $> c$ uniformly. Then,
\begin{equation}\label{eq:B_upper}
B\lesssim R^{\frac{n}{2}+C\delta}\abs{\mathbb{T}_2}.
\end{equation}

We use \eqref{eq:B_lower} and \eqref{eq:B_upper} to conclude that
\begin{equation*}
\abs{\mathbb{T}'_1(q_0)(\xi',\xi'-\xi'')}\lesssim R^{C\delta}\frac{\abs{\mathbb{T}_2}}{\lambda_1\mu_2},
\end{equation*}
which is what we wanted to prove.

\subsection{End of the Proof} \label{sec:end_Bilinear}

In this section we reap  all the bounds we have obtained. We plug \eqref{eq:nu_term} into \eqref{eq:bilinear_decoupled} to get
\begin{align*}
\sum_{q\in q(\mu_2)}\norm{\sum_{\substack{T_1\in \mathbb{T}_1'(q) \\ T_2\in\mathbb{T}_2(q)}} \phi_{T_1}\phi_{T_2}}_{L^2(q)}^2 &\le \nu R^{1-\frac{n}{2}+C\delta}\abs{\mathbb{T}_2}\sum_{q\in q(\mu_2)}\lambda_1^{-1}\abs{\mathbb{T}_1[\mu_2,\lambda_1](q)} \\
&\lesssim \nu R^{1-\frac{n}{2}+C\delta}\abs{\mathbb{T}_2}\sum_{T_1\in \mathbb{T}_1[\lambda_1,\mu_2]}\lambda_1^{-1}\sum_{\substack{q\in q(\mu_2) \\ }}\ind_{\{T_1\cap R^\delta q\neq\emptyset\}} \\
&\lesssim \nu R^{1-\frac{n}{2}+C\delta}\abs{\mathbb{T}_1}\abs{\mathbb{T}_2},
\end{align*}
This concludes the proof of \eqref{eq:bilinear_L2_pigeon}.

We interpolate the bilinear norm between the points $p'=1$ in \eqref{eq:point_p1} and $p'=2$ in \eqref{eq:bilinear_L2_pigeon} to get
\begin{equation*}
\norm{\sum_{T_1\nsim B_{R'},T_2} \phi_{T_1}\phi_{T_2}}_{L^\frac{n}{n-1}(B_{R'})}\le C_\delta (\log R)^CR^{C\delta}\nu^\frac{1}{n}\abs{\mathbb{T}_1}^\frac{1}{2}\abs{\mathbb{T}_2}^\frac{1}{2}.
\end{equation*}
This bound joins the inequalities \eqref{eq:Bil_simplified} and \eqref{eq:Inductive_Hypothesis} to yield
\begin{equation*}
\norm{\sum_{\substack{T_1\in \mathbb{T}_1 \\ T_2\in\mathbb{T}_2}}\phi_{T_1}\phi_{T_2}}_{L^\frac{n}{n-1}(B_R)}\le C_\delta(\log R)^C(K_\nu(R')+R^{C\delta}\nu^\frac{1}{n}) \abs{\mathbb{T}_1}^\frac{1}{2}\abs{\mathbb{T}_2}^\frac{1}{2};
\end{equation*}
in other words,
\begin{equation*}
K_\nu(R)\le C_\delta(\log R)^C(K_\nu(R^{1-\delta})+R^{C\delta}\nu^\frac{1}{n}).
\end{equation*}
When we iterate, we get at the $N$-th step
\begin{equation*}
K_\nu(R)\le C_\delta^N (\log R)^{NC}(K_\nu(R^{(1-\delta)^N})+NR^{C\delta}\nu^\frac{1}{n}).
\end{equation*}
We stop when $R^{(1-\delta)^N}\le \nu^{-1}< R^{(1-\delta)^{N-1}}$; the number of steps is 
\begin{equation*}
N\le -\frac{1}{\log(1-\delta)}+1\le 2\delta^{-1}.
\end{equation*}
If $r\le \nu^{-1}$, then we can average over translations of the paraboloid and apply Tao's bilinear to get $K_\nu(r)\le C_\varepsilon r^{1-\frac{n+2}{2p}+\varepsilon}\nu^\frac{1}{2}$. We have thus that
\begin{equation*}
K_\nu(R)\le C_\delta R^{C\delta}(\nu^{-1+\frac{n+2}{2n}+\frac{1}{2}}+\nu^\frac{1}{n})\le C_\delta R^{C\delta}\nu^\frac{1}{n}.
\end{equation*}
This concludes the proof of Theorem~\ref{thm:Bilinear_Modified}, which implies Theorem~\ref{thm:Bilinear}'.

\subsubsection{Additional Remarks}

We indicate here the changes we need to do for surfaces of elliptic type or the hemisphere. The argument is sufficiently robust to admit perturbations.

For surfaces of $\varepsilon$-elliptic type, the semi-norms $\norm{\partial^N\Phi}_\infty$ enter in the constants $C_\delta$ of \eqref{eq:f_outside_delta} and \eqref{eq:g_outside_delta}. Since the eigenvalues of $D^2\Phi$ are close to one, then the tubes have approximately the same length. 

The delta function in \eqref{eq:delta_delta} gets into
\begin{equation*}
\delta(\Phi(\xi_2')-\Phi(\xi_1'+\xi_2')+\Phi(\xi_1'+\xi_2'')-\Phi(\xi_2''))=\delta(\inner{A(\xi_2'-\xi_2'')}{\xi_1'})
\end{equation*} 
for some matrix $A$ with eigenvalues in $[1-\varepsilon,1+\varepsilon]$. Then $\abs{\inner{A(\xi_2'-\xi_2'')}{\xi_1'}-\inner{\xi_2'-\xi_2''}{\xi_1'}}\le C\varepsilon$, and instead of an integral over the hyperplane $H$ with normal direction $\xi_2'-\xi_2''$ that passes through $\xi_2'$, we integrate over a $(n-2)$-surface $\tilde{H}$ that lies in a $\varepsilon$-neighborhood of $H$ and passes through $\xi_2'$. 

A tube associated with a cap with center $c_\alpha$ has velocity $(-\nabla\Phi(c_\alpha),1)$. If $\tilde{P}\subset\Rs^n$ is a $(n-1)$-cone with center in a cube $q$ generated by all the lines with directions $(-\nabla\Phi(\eta),1)$ for $\eta\in\tilde{H}$, then we must verify that all the tubes coming from the separated set $S_2$ are transversal to $\tilde{P}$. In fact, notice that for any point $\xi_2'+\xi_1'\in\tilde{H}$, a vector $v$ tangent to $\tilde{H}$ satisfies the equation
\begin{equation*}
\inner{\nabla\Phi(\xi_1'+\xi_2'')-\nabla\Phi(\xi_1'+\xi_2')}{v}=0;
\end{equation*}
hence, $\inner{A(\xi_2''-\xi_2')}{v}=0$ for some matrix $A$ close to $I$. Then, the vectors normal to $\tilde{P}$ have the form $(A(\xi_2''-\xi_2'), \inner{\nabla\Phi(\xi_2'+\xi_1')}{A(\xi_2''-\xi_2')})$. If we take the inner product of these vectors with $(-\nabla\Phi(\eta_2),1)$ for $\eta_2\in\pi(S_2)$, then we get 
\begin{equation*}
\inner{A(\xi_2''-\xi_2')}{\nabla\Phi(\xi_2'+\xi_1')-\nabla\Phi(\eta_2)}=\inner{A(\xi_2''-\xi_2')}{A'(\xi_2'+\xi_1'-\eta_2)};
\end{equation*}
hence, the inner product is basically equal to $\inner{\eta_1-\eta_2}{\eta_1'-\eta_2'}$ for all the pairs $\eta_1,\eta_1'\in\pi(S_1)$ and $\eta_2,\eta_2'\in\pi(S_2)$, and $\abs{\inner{\eta_1-\eta_2}{\eta_1'-\eta_2'}}\ge c>0$, then $\tilde{P}$ is uniformly transversal to all the tubes coming from $S_2$. The estimates hold uniformly in $\varepsilon\ll 1$.

The case of the hemi-sphere is similar. The term \eqref{eq:delta_delta} is almost as simple as for the paraboloid. By symmetry, we can assume that $\xi_2'=-ae_1$ and $\xi_2''=ae_1$ for some $0<a\le\frac{1}{\sqrt{2}}+\frac{1}{10}$. Then, the $(n-2)$-surface $\tilde{H}$ is again a hyperplane $H$ with normal direction $e_1$ that passes through $\xi_2'$. The cone $\tilde{P}$ is a translation of a portion of the quadratic cone $\{\xi\mid \xi_1^2=a^2\abs{\xi}^2\}$. It is intuitively clear that the portion of the cone generated by direction from $S_1$ is uniformly transversal to tubes from $S_2$.

\subsection*{Notations}
 
\begin{itemize}
\item Relations: $A\lesssim_\epsilon B$ if $A\le C_\epsilon B$; $A\sim B$ if $A\lesssim B\lesssim A$; $A\ll 1$ if $A\le c$, where $c$ is chosen sufficiently small.
\item Various: $e(z):=e^{2\pi i z}$, $\japan{x}=(1+\abs{x}^2)^\frac{1}{2}$, $B_r(x)$ a ball of radius $r$ with center at $x$. $\dashint_M\,d\tau := \frac{1}{M}\int_M^{2M}\,d\tau$. $a+:=a+\varepsilon$ for $\varepsilon\ll 1$. $\abs{E}$ is the Lebesgue measure of a set $E\subset\Rs^n$, or the cardinality of a finite set $E$. If $T$ is a tube with main axis $l$, then $AT$ is a dilation of $T$ by a factor $A>0$ and same main axis $l$.
\item Multipliers: $m(D)f = (m\widehat{f})^\vee$, where $m$ stands for \textit{multiplier}; $Pf = m(D)f$, where $m$ is a smooth cut-off for a set of frequencies where we want to project to.
\item The operator $\Delta_\zeta:=\Delta + \zeta\cdot\nabla$ has symbol $p_\zeta(\xi):=-\abs{\xi}^2+2i\zeta\cdot\nabla$ and characteristic $\Sigma_\zeta:=\{\xi\mid p_\zeta(\xi)=0\}$.
\item $\zeta(U,\tau):=\tau(Ue_1-iUe_2)$, where $\{e_i\}$ is the canonical basis, $\tau\ge 1$ and $U\in O_d$ is a rotation.
\item $\norm{u}_{\dot{X}^b_\zeta}^2:= \int \abs{p_\zeta(\xi)}^{2b}\abs{\widehat{u}(\xi)}^2\,d\xi$.
\item $\norm{u}_{X^b_{\zeta,\sigma}}^2:= \int (\abs{p_\zeta(\xi)}+\sigma)^{2b}\abs{\widehat{u}(\xi)}^2\,d\xi$ for $\sigma>0$; $\norm{u}_{X^b_\zeta}=\norm{u}_{X^b_{\zeta,\abs{\zeta}}}$.
\item Sobolev-Slobodeckij spaces: For $1\le p<\infty$, $W^{s,p}(\Rs^d)$ is the space of distributions $f$ such that 
\begin{align}
\norm{f}_{s,p}&:= \sum_{\abs{\alpha}\le s}\norm{D^\alpha f}_p\quad \text{for } s \text{ integer.} \\
\norm{f}_{s,p}&:= \norm{P_{\le 1}f}_p+\Big(\sum_{k>0}2^{skp}\norm{P_kf}_p^p\Big)^\frac{1}{p}<\infty\quad \text{for } 0<s\neq \text{integer.}
\end{align}
For a domain $\Omega\subset\Rs^d$, we define $W^{s,p}(\Omega):=\{f|_\Omega \mid f\in W^{s,p}(\Rs^d)\}$. The space $\accentset{\circ}{W}^{s,p}(\Omega)$ is the completion in $W^{s,p}(\Rs^d)$ of test functions $D(\Omega):=\{\varphi\in C^{\infty}(\Omega)\mid \supp \varphi\Subset\Omega\}$. For further details, see \textit{e.g.} \cite{MR781540,MR884984}.
\item $\ext{f}(x):=\int_{\Rs^{n-1}} f(\xi)e(\inner{x'}{\xi}+x_n\varphi(\xi))\,d\xi$, where $S$ is the graph of $\varphi$ and $(x',x_n)\in\Rs^n$.
\end{itemize}

\bibliographystyle{plain}
\bibliography{../Calderon_Problem,../Restriction}

\end{document}